\documentclass{article}

\usepackage{PRIMEarxiv}

\usepackage[utf8]{inputenc} 
\usepackage[T1]{fontenc}    
\usepackage{hyperref}       
\usepackage{url}            
\usepackage{booktabs}       
\usepackage{amsfonts}       
\usepackage{nicefrac}       
\usepackage{microtype}      
\usepackage{lipsum}
\usepackage{fancyhdr}       
\usepackage{graphicx}       
\graphicspath{{media/}}     

\usepackage{amsmath,amssymb,amsfonts}         
\usepackage{tikz}
\usepackage{graphicx}
\usepackage{booktabs}

\newtheorem{lemma}{Lemma}
\newtheorem{proof}{Proof}
\newtheorem{theorem}{Theorem}

\title{Cascading failures: dynamics, stability and control}

\author{
  Stefanny Ramirez, Maaike Odijk\\
  Jan C. Willems Center for Systems and Control, \\
  Engineering and Technology Institute Groningen, \\
  Faculty of Science and Engineering, University of Groningen \\
  Nijenborgh 4, 9747 AG Groningen, The Netherlands\\
  \texttt{s.g.ramirez.juarez@rug.nl} \\
  \AND
  Dario Bauso \\
  Jan C. Willems Center for Systems and Control, \\
  Engineering and Technology Institute Groningen, \\
  Faculty of Science and Engineering, University of Groningen\\
  Nijenborgh 4, 9747 AG Groningen, The Netherlands \\
  Dipartimento di Ingegneria, Università di Palermo \\
  viale delle Scienze 90128 Palermo, Italy \\
  \texttt{d.bauso@rug.nl} \\
}

\begin{document}
\maketitle

\begin{abstract}
We develop a dynamic model of cascading failures in a financial network whereby cross-holdings are viewed as feedback, external assets investments as inputs and failure penalties as static nonlinearities. We provide sufficient milder and stronger conditions for the system to be a positive one, and study equilibrium points and stability. Stability implies absence of cascades and convergence of market values to constant values. We provide a constructive method for control design to obtain stabilizing market investments in the form of feedback-feedforward control inputs.
\end{abstract}


\section{Introduction.}
\label{sec:introduction}

Cascading failures in financial networks is the phenomenon by which a little shock at one node propagates throughout the network. Nodes represent organizations and edges represent interdependent cross-holding. The shocks are financial failures and take place whenever the market value of an organization falls below a threshold. The market value of an organization depends on their investments on external assets and on shares of other organizations. 

\subsection{Main contribution.}

We highlight the following main contributions. First, we extend the model in~\cite{Eisenberg} and \cite{elliott2014financial} to a dynamic setting whereby  market values change with time based on the current investments and current cross-holdings. The model has a feedback structure (cross-holdings) with external inputs (assets investments) and static nonlinearities (discontinuous failure penalties).

Second, we provide a detailed analysis within the context of positive systems. In particular, we provide sufficient conditions for the system to be a positive one and provide a physical interpretation in terms of cascading failure network resilience. Under such conditions, the system admits steady-state solutions (equilibrium points).

Third, we analyze conditions for such equilibrium points to be asymptotically stable by using the theory of Lyapunov stability. Stability implies absence of cascades and convergence of market values to a constant value. Furthermore, we relax the aforementioned sufficient conditions and prove that equilibrium points may exist under milder conditions on the network connectivity.

Fourth, we extend the stability analysis to the case of large-network and reinterpret the previous stability result in terms of centrality and degree connectivity of the network. 

Fifth, we provide a constructive method for control design to obtain stabilizing market investments in the form of feedback-feedforward control inputs. The feedforward control blocks the propagation of the failures while the feedback control makes the market value dynamics stable and converging to a stationary value.

\subsection{Related Literature.}
We build on the model in~\cite{elliott2014financial}, and we extend it to a dynamic setting in the same spirit as in \cite{Calafiore}. Similar models as the one in~\cite{elliott2014financial} can be found in the previous literature, see for instance~\cite{Eisenberg}. In~\cite{Eisenberg} the authors define the equity value in terms of debts and payments between institutions, and the cascade of failures is propagated when a failed institution reduces its payment to its creditors, which can generate the failure of at least one of these creditors. On the other hand, in~\cite{elliott2014financial} the equity value is defined in terms of shares held by the organizations as well as by external shareholders, and the propagation of failure is greatly affected by a failure cost. 

Dynamic threshold models are a classic of the literature on cascading failures and for further details we refer the reader to the seminal works~\cite{ADO2012}, \cite{Glasserman}, \cite{Granovetter}, \cite{Jackson}, \cite{rossi2017threshold}, \cite{Massai} and \cite{Rakesh21}. Dynamical models to study and mitigate the systemic risk in financial networks are developed in~\cite{Capponi} and \cite{Chen21}. A dynamic setting allows us to use tools from systems and control to examine the stability of a system with multiple equilibrium points. In particular, our model uses quaternion transformation matrices and this allows to view it as a positive system~\cite{farina2000positive}, and analyse the transient of the system when there exists any perturbation that drives it away  from an equilibrium point. Furthermore, thanks to this dynamic framework, in our work we design a feedback-feedforward control to determine the market investments that stabilize the system and stop the cascade of failures.

The literature on cascading failures is vast and grows in several directions some of which depart from the scope of the current paper. For instance, trading uncertainty and market frictions are the main focus in~\cite{Douglas}. The vulnerability to failures of interconnected economic networks, also referred to as network of networks, is studied in~\cite{havlin2015cascading}. Contagious links and their long-term impact on the resilience of random networks  are studied in~\cite{Amini}. Stochastic models of cascading failures to determine the moment of the first failure in financial networks are developed in~\cite{Swift2008} and \cite{lee2020new}. The underlying perspective in this paper follows  the idea of viewing ``economists as engineers'' in the context of market design~\cite{Roth} and \cite{Smith}. 

Previous attempts of assimilating cascading failures to epidemic models seem to fall short of satisfactory answers when it comes to understanding pairwise interactions, which is the main focus in~\cite{Vega-Redondo}. The use of contagion models to understand the dynamics of financial and biological networks is analyzed in~\cite{Peckham}. In~\cite{Barja19} the authors adapt the framework of epidemiology models to default propagation in financial networks. In this sense,  neighbour correlation may have an impact on preventing exponential growth of contagion~\cite{Morris}. Here the focus is on the different behaviour observed in strong and weak social links. More specifically, a main factor of limitation of contagion is provided by  strong links for which the neighbor relations are in general transitive (friends of friends are usually friends) in accordance to the definition in~\cite{Granovetter}.  Another factor limiting contagion are random links between otherwise distant nodes. This is observed for instance whenever  the contagion requires simultaneous exposure to multiple sources of activation~\cite{Centola}. Influence maximization is also a challenging factor in networks with threshold models. The idea here is to find the smallest set of nodes with maximal aggregated influence~\cite{CYZ2010},~\cite{GLL2011}. A game-theoretic perspective of contagion assimilates it to a coordination game~\cite{Ellison}. Cascades of failures follow the rules of contagion in pairwise interactions between nodes in a network. This places emphasis on certain network properties such as  connectivity or centrality. While connectivity allows the failure to propagate, a high connectivity is proven to limit the diffusion due to the impact of high-degree nodes which are very stable~\cite{Lelarge}. 

\subsection{Organization of the manuscript.}

This manuscript is organized as follows. In Section~\ref{sec:model}, we develop the model. 
In Section~\ref{sec:analysis}, we cast the model within the theory of positive systems and provide the equilibrium point analysis including stability. In Section~\ref{sec:num}, we perform numerical simulations.  Finally in Section~\ref{sec:conc} we provide conclusions.

\section{Model.}
\label{sec:model}

The financial network involves a set of companies $N = \{1,\ldots,n\}$ and a set of primitive assets $M=\{1,\ldots,m\}$. Asset $h\in M$ is characterized by the market price $p_h\in \mathbb R$. Company $i\in N$ is characterized by an equity value $V_i \in \mathbb R$ and a market value $v_i \in \mathbb R$. Company $i$ owns a fraction $c_{ij} \in [0,1] $ of company $j$ for each pair $i,j\in N$, $i\not = j$. No company has shares in itself, i.e., $c_{ii} = 0$  $\forall i \in N$. The fraction of company $i$ which is not owned by other companies is  $\hat{c}_{ii} := 1 - \sum_{j \in N, \, i\not = j}{c_{ji}}$, $\hat{c}_{ii}>0$. Likewise, company $i$ has a share $D_{ih}$ of asset $h$, for each $i \in N$ and $h\in M$. To introduce a compact notation, denote by $V=[V_i]_{i \in N} \in \mathbb R^n$ and $v=[v_i]_{i \in N} \in \mathbb R^n$ the vector of equity values and market values, respectively. Also let $C:=[c_{il}]_{i,l\in N} \in \mathbb R^{n\times n}$ be the matrix of cross-holding, $D:=[D_{ih}]_{i\in N,\, h \in M}$ be the matrix of shares on primitive assets, and $p=[p_h]_{h\in M}$ the vector of market prices for the primitive assets. Let also denote by $\hat C:=diag(\hat c_{ii})\in \mathbb R^{n\times n}$ the diagonal matrix containing all the fractions of companies not owned by other companies. When the market value $v_i$ of company $i$ goes below a given threshold $\underline v_i$ the company `fails' and incurs a fixed nonnegative failure cost $\beta_i \in \mathbb R_+$. The equity value is determined by the value of the company primitive assets, the value of its shares on other companies and the failure cost according to the following discrete-time dynamic model:
 \begin{equation}
    \begin{cases} V(t+1) = CV(t) + Dp - \phi(v),
                \\v(t) = \hat CV(t),
    \end{cases}
    \label{eq:equity_value_model_1}
\end{equation}
where $\phi(v) = [\beta_i I_{v_i < \underline{v}_i}(v_i)]_{i\in N} \in \mathbb R^n$ and $I_{v_i < \underline{v}_i}(v_i)=1$ if $v_i < \underline v_i$ and $I_{v_i < \underline{v}_i}(v_i)=0$ otherwise. 

To investigate the properties of (\ref{eq:equity_value_model_1}) in the context of positive systems let us rewrite it in terms of the new variables
\begin{equation}\label{eit}
    x_i(t) = v_i(t)-\underline{v}_i, \quad \mbox{for all $i \in N$}.
\end{equation}

Company $i$ fails if $x_i(t)<0$, and does not fail if $x_i(t)\geq 0$. Therefore the orthant that contains the current state $x(t)=[x_i(t)]_{i \in N}$ determines also which companies have failed at time $t$. When an orthant is invariant, namely when the state trajectory does not transition between different orthants, the sign of $x_i(t)$ does not change over time and there is no cascade of failure.

The term $\phi(v)$ can be modelled as a bang-bang input. Let us consider the offset $U_i^{(offset)}=\frac{\beta}{2}$ and let us introduce the function
\begin{equation}
    U_i(x) := 
    \begin{cases}  \frac{\beta}{2} , &\text{if}  \textcolor{white}{...} x_i<0, 
                \\-\frac{\beta}{2}, &\text{if}  \textcolor{white}{...} x_i\geq0. 
    \end{cases}
    \label{eq:ON-OFF}
\end{equation}
Let the following vectors be given $U(x):=[U_i(x)]_{i \in N} \in \mathbb R^n$ and $U^{(offset)}:=[U_i^{(offset)}]_{i \in N} \in \mathbb R^n$. We can rewrite $\phi(v)$ as follows:
\begin{equation}\label{eq:bv}
    \phi(v) = U(x) + U^{(offset)} = U(x) + \mathbf{1}\frac{\beta}{2}.
\end{equation}

The term $\phi(v)$ as in (\ref{eq:bv}) represents an offset of the penalty term and can be interpreted as a static nonlinearity. From (\ref{eq:equity_value_model_1}), (\ref{eit}) and (\ref{eq:bv}), after  denoting $\underline v :=[\underline v_i]_{i\in N}$, dynamics~(\ref{eq:equity_value_model_1})  can be reformulated as a dynamics on the error $x(t)$ as: 
\begin{align}
    \begin{array}{lll} x(t+1) &= v(t+1) - \underline v = \hat C V(t+1) - \underline v = \hat C \Big[ C V(t) + Dp  - \phi (v) \Big] - \underline v \\
    & =\hat{C}C\hat{C}^{-1}x(t) + (\hat{C}C\hat{C}^{-1}-I)\underline{v} + \hat{C}(Dp - U(x(t)) - \mathbf{1}\frac{\beta}{2}).
    \end{array}
    \label{eq:error_model_1}
\end{align}

In the last equality we have used the fact that  
$$V(t)= \hat C^{-1} v(t) = \hat C^{-1} (x(t) +\underline v).$$

Let us enumerate the $2^{n}$ orthants according to a specific but random order by using the index  $r=1,\ldots,\hat{r}$ where $\hat{r}:=2^n$. 
Let us denote by
$\mathcal O^{(r)}$ the $r$th orthant. For each orthant $\mathcal O^{(r)}$ $r=1,\ldots,\hat{r}$ consider the following quaternion matrix $J^{(r)}=[j_{il}^{(r)}]_{i,l\in N}$ where $j^{(r)}_{il} := 0$, for all $i,l\in N, \, i \not = l,$ and for any generic $x$ in $\mathcal O^{(r)}$
\begin{equation}
    j^{(r)}_{ii} := 
    \begin{cases} -1 ,& \textcolor{white}{....}\text{if} \textcolor{white}{.} x_i<0, 
                \\1 ,& \textcolor{white}{....}\text{if}  \textcolor{white}{.} x_i\geq0.
    \end{cases} 
    \label{eq:transform}
\end{equation}
This transformation allows to turn any orthant into the positive orthant. In other words, for any initial state $x(0)$ in the $r$th orthant $\mathcal O^{(r)}$ we have a corresponding initial state in the positive orthant given by 
\begin{equation}\label{eq:transformation}
    X(0) = J^{(r)}x(0)\in \mathbb R_+^n.
\end{equation}
The transformation matrix can be used to define any generic orthant as follows $\mathcal O^{(r)}:=\{x \in \mathbb R^n|\, J^{(r)} x \geq 0\}$ for all $r= 1,\ldots, \hat r$. Using the quaternion matrix $J^{(r)}$, dynamics (\ref{eq:error_model_1}) can be written in terms of the new variable $X(t)$ as:
\begin{equation}
      \begin{array}{ll} X(t+1) = J^{(r)}\hat{C}C\hat{C}^{-1}J^{(r)}X(t) + J^{(r)}\Big[(\hat{C}C\hat{C}^{-1}-I)\underline{v} + \hat{C}\Big(Dp - U(J^{(r)}X(t)) - \mathbf{1}\frac{\beta}{2} \Big)\Big].
    \end{array}
    \label{eq:model_etilde}
\end{equation}

After introducing the new matrix 
\begin{equation}
\begin{array}{lll}
\label{eqn:structurematrixA^r}
    A^{(r)} :=  [a_{ij}]_{i,j=1,\ldots,n}
    :=  J^{(r)}\hat{C}C\hat{C}^{-1}J^{(r)}
    =  \begin{bmatrix}
    0  & \dots & j_{nn}\frac{1}{\hat{c}_{nn}}c_{1n}j_{11}\hat{c}_{11}\\
    j_{11}\frac{1}{\hat{c}_{11}}c_{21}j_{22}\hat{c}_{22}  & \dots & j_{nn}\frac{1}{\hat{c}_{nn}}c_{2n}j_{22}\hat{c}_{22}\\
    \vdots  & \ddots & \vdots\\
    j_{11}\frac{1}{\hat{c}_{11}}c_{n1}j_{nn}\hat{c}_{nn}  & \dots & 0\\
    \end{bmatrix},
\end{array}
\end{equation}
and vector 
\begin{equation}\label{br}
b^{(r)}(X):=J^{(r)}\Big[(\hat{C}C\hat{C}^{-1}-I)\underline{v} + \hat{C}\Big(Dp - U(J^{(r)}X) - \mathbf{1}\frac{\beta}{2} \Big)\Big]
\end{equation}
the model can be written in compact form as:
\begin{equation}
    X(t+1) = A^{(r)}X(t)+b^{(r)}(X(t)).
    \label{eq:error_model_33}
\end{equation}

Equivalently we can write the above in input-output form as \begin{equation}
    \begin{cases} X(t+1) = A^{(r)}X(t)+b^{(r)}(X(t)),\\
    v(t)= J^{(r)}X(t) + \underline{v}.
    \end{cases}
    \label{eq:error_model_3}
\end{equation}

Figure \ref{fig:block_error_tilde} depicts a block diagram representation of (\ref{eq:error_model_3}) in feedback-form. 
\begin{figure} [h]
\centering
\begin{tikzpicture}[scale=0.70]
\begin{scope}[shift={(0,0)},scale=1]
 \draw [rounded corners, very thick](-0.9,0) rectangle (6.6,1.6);
 \node at (2.9,.8) {\footnotesize 
 \(\displaystyle 
 \begin{array}{ll}
 X(t+1) = A^{(r)}X(t)+b^{(r)}(X(t))\\
 v(t)= J^{(r)}X(t) + \underline{v}
 \end{array}
 \)
 };
 \draw[->, very thick](6.6,.7)--(7.6,.7);\node at (7.2,1.1) {\small $v(t)$};
 \draw[->,rounded corners, very thick] (7.1,.7)--(7.1,-1.5)--(-6.2,-1.5)--(-6.2,.2);\node at (-6.8,-0.4) {\small $v(t)$};
 \draw[->,rounded corners, very thick] (-5.4,.7)--(-5.4,-0.8)--
(-4,-0.8);
\draw[->,rounded corners, very thick]
(.1,-0.8) --(1.8,-0.8)--(1.8,0);\node at (-2,-0.8) {$X(t) = J^{(r)}x(t)$};
 \draw[black,rounded corners, very thick ] (-4,-1.3) rectangle (0.1,-0.3);
\draw[very thick ] (-6.2,.7) circle (.43cm);\node at (-6.2,0.5) {\tiny $+$};\node at (-6.4,0.7) {\tiny $-$};
\draw[->, very thick](-7.5,.7)--(-6.7,.7);\node at (-7.2,1.0) {$\underline v $};
\draw[->, very thick ](-5.8,.7)--(-4.9,.7);\node at (-5.4,1.2) {\footnotesize $x(t)$};
\draw[black,rounded corners, very thick ] (-4.9,0.2) rectangle (-3.9,1.2);\node at (-4.4,.7) {\footnotesize $U(.)$};
\draw[->, very thick ](-3.9,.7)--(-3.1,.7);
\draw [very thick ](-2.7,.7) circle (.43cm);\node at (-2.9,0.7) {\tiny $-$};\node at (-2.7,0.99) {\tiny $-$};
\draw [->, very thick ] (-2.7,2.7)-- (-2.7,1.2);\node at (-2.4,2.3) {$\frac{\beta}{2}$};
\draw [->, very thick ] (-2.3,.7)-- (-0.95,.7);\node at (-1.7,1.2) {\footnotesize $-\phi(v)$};
\draw [->, very thick] (2.8,3.0)-- (2.8,1.7);\node at (3.3,2.4) {$Dp$};
\end{scope}
\end{tikzpicture}
\vspace{5mm}
\caption{Block diagram representation of (\ref{eq:error_model_3}) in feedback-form.} 
\label{fig:block_error_tilde}
\end{figure}
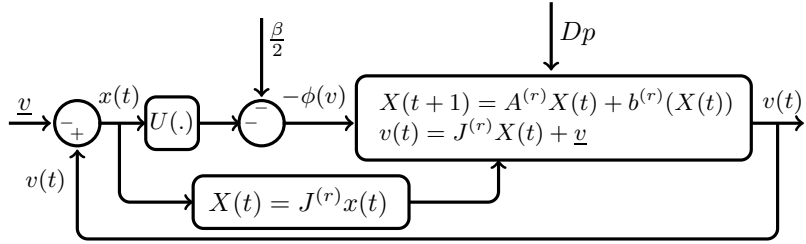

\section{Analysis.}
\label{sec:analysis}

We recall that a square matrix $A \in \mathbb R^{n \times n}$ is nonnegative if the entries $a_{ij}\geq  0$ for all $i,j=1,\ldots n$. We denote $A \geq 0$ and $v \geq 0$ for a nonnegative matrix $A$ and vector $v$. If the inequalities are strict, namely $a_{ij}>0$ then the matrix is said  positive. We denote $A > 0$ and $v > 0$  for a positive  matrix A and vector $v$.

An affine system of type $x(k+1) = A x(k) +b$ is positive if $A\geq 0$ and $b\geq 0$. For a positive system it holds that $x(0) \geq 0$ implies $x(k)\geq 0$ for all $k=1,2,\ldots$. Also, if $A$ is convergent then $\lim_{k \rightarrow  \infty} x(k) = (I-A)^{-1} b \geq 0$. We recall that $A$ is convergent if all eigenvalues are in the unit disc, namely $|\lambda|<1$ where $\lambda \in spec(A)$ and $spec(A)$ denotes the spectrum of $A$ (set of all eigenvalues).


Consider the generic orthant $\mathcal O^{(r)}$ and let $\mathcal P^{(r)}:=\{ i\in N: \, j_{ii}=1 \}$ and $\mathcal N^{(r)}:=\{ i\in N: \, j_{ii}=-1 \}$. Evidently it holds $N = \mathcal P^{(r)} \cup \mathcal N^{(r)}$ and $ \mathcal P^{(r)}$, $ \mathcal N^{(r)}$ are a partition of~$N$. Let us define $b^{(r)}_+(X):= [b^{(r)}_i(X)]_{i \in \mathcal P^{(r)}} \in \mathbb R^{|\mathcal P^{(r)}|}$. In other words, $b^{(r)}_+(X)$ is the vector obtained by extracting only the components $i \in \mathcal P^{(r)}$ from $b^{(r)}(X)$.   

The first result is about the sign of the vector function $b^{(r)}(X) \in \mathbb R^n$ as defined in (\ref{br}).
\begin{lemma}\label{trm:bk}
Let $x \in \mathcal O^{(r)}$ be given, and consider $X = J^{(r)} x$. We have $b^{(r)}(X) \geq 0$ if  for all $i=1,\ldots,n$: 
\begin{eqnarray}\label{constraintb>0_1}
             (Dp)_{i} \leq \Big((\hat{C}^{-1}-C\hat{C}^{-1})\underline{v}\Big)_i +\beta \qquad \text{ if } i \in \mathcal N^{(r)},
                \\
                \label{constraintb>0_1bis}
                (Dp)_{i} \geq \Big((\hat{C}^{-1}-C\hat{C}^{-1})\underline{v}\Big)_i  \qquad \text{if }  i \in \mathcal P^{(r)}.
\end{eqnarray}
\end{lemma}
\begin{proof}
Let us start by considering that for the $i$th component we have $b^{(r)}_i(X)\geq 0$ if and only if
\begin{align}\label{eq:brdef>0}
    \begin{array}{lll}
        b^{(r)}_i(X)&=\Big(J^{(r)}\Big[(\hat{C}C\hat{C}^{-1}-I)\underline{v} 
        + \hat{C}\Big(Dp - U(J^{(r)}X) - \mathbf{1}\frac{\beta}{2}\Big)\Big]\Big)_i \\
        &=j^{(r)}_{ii} 
        \Big((\hat{C}C\hat{C}^{-1}-I)\underline{v} + \hat{C}\Big(Dp - U(J^{(r)}X) - \mathbf{1}\frac{\beta}{2}\Big)\Big)_i
        \geq 0.
        \end{array}
\end{align}
For the first part of the proof let $j^{(r)}_{ii} = -1$. From $j^{(r)}_{ii} = -1$, inequality (\ref{eq:brdef>0}) is true if and only if 
$$\Big((\hat{C}C\hat{C}^{-1}-I)\underline{v} + \hat{C}\Big(Dp -  U(J^{(r)}X) - \mathbf{1}\frac{\beta}{2}\Big)\Big)_i\leq 0.$$
    
From $\hat C$ being positive and diagonal the pre-multiplication of the vector in the left-hand-side by $\hat C^{-1}$ keeps the sign of the elements of the vector. Hence, the above inequality is true if and only if 
    \begin{equation}
        \begin{array}{cc}
             \Big(\hat{C}^{-1} (\hat{C}C\hat{C}^{-1}-I)\underline{v} + \hat{C}^{-1} \hat{C}(Dp - U(J^{(r)}X) - \mathbf{1}\frac{\beta}{2}) \Big)_i  \\
             =\Big( (C\hat{C}^{-1}-\hat{C}^{-1})\underline{v} + Dp - U(J^{(r)}X) - \mathbf{1}\frac{\beta}{2} \Big)_i\leq 0.
        \end{array}    
    \end{equation}
From the above we obtain  
\begin{equation}\label{eq:dpbr1}
         (Dp)_{i} \leq \Big(-(C\hat{C}^{-1}-\hat{C}^{-1})\underline{v}+U(J^{(r)}X) + \mathbf{1}\frac{\beta}{2}\Big)_i.
\end{equation}
     
From $j^{(r)}_{ii} = -1$ we have $\Big(U(J^{(r)}X) + \mathbf{1}\frac{\beta}{2}\Big)_i=\beta$  and  (\ref{constraintb>0_1}) is proven. 

For the second inequality (\ref{constraintb>0_1bis}), 
let $j^{(r)}_{ii} = 1$.   Then inequality (\ref{eq:brdef>0}) is true if and only if 
    $$\Big((\hat{C}C\hat{C}^{-1}-I)\underline{v} + \hat{C} \Big(Dp - U(J^{(r)}X) - \mathbf{1}\frac{\beta}{2}\Big) \Big)_i\geq 0.$$
The above inequality is true if and only if 
    \begin{equation}
        \begin{array}{cc}
             \Big(\hat{C}^{-1} (\hat{C}C\hat{C}^{-1}-I)\underline{v} + \hat{C}^{-1} \hat{C}\Big(Dp - U(J^{(r)}X) - \mathbf{1}\frac{\beta}{2} \Big) \Big)_i  \\
             =\Big( (C\hat{C}^{-1}-\hat{C}^{-1})\underline{v} + Dp - U(J^{(r)}X) - \mathbf{1}\frac{\beta}{2} \Big)_i\geq 0.
        \end{array}    
\end{equation}
As $j_{ii}^{(r)}=1$ then we have $U(J^{(r)}X) + \mathbf{1}\frac{\beta}{2}=0$ and for the above we obtain  \begin{equation}\label{eq:dpbr2}
         (Dp)_{i} \geq \Big(-(C\hat{C}^{-1}-\hat{C}^{-1})\underline{v}\Big)_i,
     \end{equation}
     and  (\ref{constraintb>0_1bis}) is proven. 
$\square$
\end{proof}

The next result addresses the case where not all companies are inter-dependent. In terms of the interdependent graph, this means that such graph is not completely connected. 

In preparation to the next result, let  $\hat{A}^{(r)}:=[A_{i\bullet}^{(r)}]_{i \in \mathcal P^{(r)}} $. Furthermore, let $\hat{A}^{(r)}_+:=[a_{ij}^{(r)}]_{i,j \in \mathcal P^{(r)}} $ and $\hat{A}^{(r)}_-:=[a_{ij}^{(r)}]_{i\in \mathcal P^{(r)}, \, j\in \mathcal N^{(r)}} $. In other words we can partition $\hat{A}^{(r)} = [\hat{A}^{(r)}_+ | \hat{A}^{(r)}_-]$. 
Also,  let $X_+(t) := [X_i(t)]_{i \in \mathcal P^{(r)}}$ and $X_-(t) := [X_i(t)]_{i \in \mathcal N^{(r)}}$. Evidently we can partition $X(t) := [X_+(t)^T \, | X_-(t)^T ]^T$.

\begin{lemma} \label{trm:pos_analysis_case4}    Matrix $A^{(r)}$ is a positive matrix, i.e. 
\begin{equation}\label{cil}
    c_{il}=0, \mbox{for all $i,l \in N$: $j^{(r)}_{ii}\neq j^{(r)}_{ll}$ }.
    \end{equation}
    Furthermore the submatrix $\hat{A}^{(r)}:=[A_{i\bullet}^{(r)}]_{i \in \mathcal P^{(r)}} $ is a positive matrix if \begin{equation}\label{cilsub}
    c_{il}=0, \mbox{for all $i,l \in N$: $i\in \mathcal P^{(r)}$, $l \in \mathcal N^{(r)}$}.
    \end{equation}
\end{lemma}

\begin{proof}
     The entries of matrix $A^{(r)}$ are defined by $a^{(r)}_{il}=j^{(r)}_{ii}\hat{c}_{ii}c_{il}\frac{1}{\hat{c}_{ll}}j^{(r)}_{ll} \textcolor{white}{...} \forall i\neq l$. If for every $j^{(r)}_{ii}, j^{(r)}_{ll} \in \{-1,1\}$ where $j^{(r)}_{ii}\neq j^{(r)}_{ll}$ the entry of the cross-holdings matrix is $c_{il}=0$, then $a^{(r)}_{il}=j^{(r)}_{ii}\hat{c}_{ii}c_{il}\frac{1}{\hat{c}_{ll}}j^{(r)}_{ll}=-1\times\hat{c}_{ii}\times0\times \frac{1}{\hat{c}_{ll}}\times 1 = 0\geq0$. In conclusion, matrix $A^{(r)}$ is a positive matrix for as long as $c_{il}=0$ for every $j^{(r)}_{ii}, j^{(r)}_{ll} \in \{-1,1\}$ where $j^{(r)}_{ii}\neq j^{(r)}_{ll}$. To prove (\ref{cilsub}) we simply apply the same reasoning only to the generic row $i$ where $j_{ii}=1$. 
    $\square$
\end{proof}

For any given orthant $\mathcal O^{(r)}$, define the subspace
$\mathcal S^{(r)}:=\{\xi \in \mathbb R^n| \,  \xi_i\geq 0, \forall i \in \mathcal P^{(r)}\}$. In simple terms, the subspace $\mathcal S^{(r)}$ is the set of points  where all positive components of a vector in $\mathcal O^{(r)}$ remain positive (but negative components can become positive). In other words, for a given $\mathcal O^{(r)}$, $\mathcal S^{(r)}= \cup_{k, \mathcal P^{(k)} \supset \mathcal P^{(r)}} \mathcal O^{(k)} $.

We are ready to establish conditions under which  $\mathcal S^{(r)}$ is positively invariant.

\begin{theorem}\label{thm1}
Let dynamics (\ref{eq:error_model_33}) be given where $X(0) \in \mathbb R_+^n$ and $x=J^{(r)} X(0)$ is in $\mathcal O^{(r)}$. If (\ref{constraintb>0_1})-(\ref{constraintb>0_1bis}) and (\ref{cil}) hold true, then we have
\begin{equation}\label{pos}
    X(0) \geq  0 \Longrightarrow X(t) \geq 0, \quad \mbox{for all $t=0,1,2,\ldots$}
\end{equation} 
Similarly, if (\ref{constraintb>0_1bis}) and (\ref{cilsub}) hold true, then we have for all $i \in \mathcal P^{(r)}$
\begin{equation}\label{pos_sub}
    X_i(0) \geq  0 \Longrightarrow 
    X_i(t) \geq 0, \quad \mbox{for all $t=0,1,2,\ldots$}
\end{equation} 

\end{theorem}

\begin{proof}
From Lemmas \ref{trm:bk} and \ref{trm:pos_analysis_case4}, if (\ref{constraintb>0_1})-(\ref{constraintb>0_1bis}) and (\ref{cil}) hold true, then we have that $A^{(r)} \geq 0$ and $b^{(r)} \geq 0$. This means that (\ref{eq:error_model_33}) is a positive affine system and this in turn implies that if the system is initialized in the positive orthant, the trajectory will remain confined in the positive orthant. 

For the second part of the proof,  we can write
$$X_+(t+1) = A^{(r)}_+ X_+(t) + A^{(r)}_- X_-(t) + b^{(r)}_+(X(t)).$$

From (\ref{cilsub}) we have that $A^{(r)}_-$ is a null matrix and therefore the above is equivalent to 
$$X_+(t+1) = A^{(r)}_+ X_+(t) + b^{(r)}_+(X(t)) \geq 0.$$

From (\ref{constraintb>0_1bis}) we have $b^{(r)}_+(X(t)) \geq 0$. It also holds $A^{(r)}_+$ and therefore we obtain (\ref{pos_sub}).

This concludes our proof. 
$\square$
\end{proof}

Condition (\ref{pos}) means that failures do not propagate and companies who fail at the beginning will not come back to the market.
Condition (\ref{pos_sub}) means that failures do not propagate but companies who fail at the beginning can come back to the market. In both cases $\mathcal S^{(r)}$ is positively invariant.

In the following result we establish conditions for stability and obtain the unique equilibrium point in the positive orthant. 

\begin{theorem} \label{trm:analysis_equilirbium}
    Consider dynamics (\ref{eq:error_model_33}), and let (\ref{constraintb>0_1})-(\ref{constraintb>0_1bis}) and (\ref{cil}) hold true. Also let 
     \begin{equation}\label{eq:suffcondition}
        0\leq\sum_{l = 1, \\l\neq i}^n   \hat{c}_{ii}c_{il}\frac{1}{\hat{c}_{ll}} < 1.
    \end{equation}
Then $X(t)$ converges to the following unique equilibrium point:
\begin{equation}   
\label{eq:trm6stabilityeq}
\begin{array}{lll}
    X^{*} = (I-J^{(r)}\hat{C}C\hat{C}^{-1}J^{(r)})^{-1} 
    \cdot
    \Big[J^{(r)}\Big((\hat{C}C\hat{C}^{-1}-I)\underline{v}
    + \hat{C}(Dp - U(J^{(r)}X^{*}) - \mathbf{1}\frac{\beta}{2})\Big)\Big]. 
    \end{array}
\end{equation}
\end{theorem}

\begin{proof}
First let us note that since (\ref{constraintb>0_1})-(\ref{constraintb>0_1bis}) and (\ref{cil}) hold true from Theorem \ref{thm1} we have that system (\ref{eq:error_model_33}) is positive. This means that the state trajectory never leaves the positive orthant, namely condition  (\ref{pos}) holds true. Within each single orthant dynamics (\ref{eq:error_model_3}) is linear and therefore there exists one unique equilibrium point. To obtain such an equilibrium point, note that from (\ref{eq:model_etilde}) we have that
    \begin{equation}\label{proof61000}
    \begin{array}{lll}
        X(t+1) = J^{(r)}\hat{C}C\hat{C}^{-1}J^{(r)}X(t) + J^{(r)}\Big[(\hat{C}C\hat{C}^{-1}-I)\underline{v} + \hat{C}(Dp - U(J^{(r)}X(t)) - \mathbf{1}\frac{\beta}{2})\Big].
        \end{array}
    \end{equation}
    
At the equilibrium it must hold: 
    \begin{equation}\label{proof62}
        X(t+1) = X(t) = X^{*}.
    \end{equation}
    
Now by substituting (\ref{proof62}) in (\ref{proof61000}), we have 
    \begin{equation}\label{proof6100000}
    \begin{array}{lll}
        X^* = J^{(r)}\hat{C}C\hat{C}^{-1}J^{(r)}X^* + J^{(r)}\Big[(\hat{C}C\hat{C}^{-1}-I)\underline{v} + \hat{C}(Dp - U(J^{(r)}X^*) - \mathbf{1}\frac{\beta}{2})\Big].
        \end{array}
    \end{equation}
    Taking all linear terms involving $X^{*}$ to the left-hand-side we obtain
        \begin{equation}\label{proof63}
        \begin{array}{lll}
            (I-J^{(r)}\hat{C}C\hat{C}^{-1}J^{(r)}) X^{*} =  
            J^{(r)}\Big[(\hat{C}C\hat{C}^{-1}-I)\underline{v} + \hat{C}\Big(Dp - U(J^{(r)}X^{*}) - \mathbf{1}\frac{\beta}{2}\Big)\Big]. 
            \end{array}
    \end{equation}
By pre-multiplying by $(I-J^{(r)}\hat{C}C\hat{C}^{-1}J^{(r)})^{-1}$ both left and right-hand side we obtain the equilibrium point in \ref{eq:trm6stabilityeq}.
    
Now it is left to prove that such an equilibrium is asymptotically stable, namely that it holds $\lim_{t \rightarrow  \infty} X (t) = X^* \geq 0$. To do this, it suffices to show that the matrix $A^{(r)}:=J^{(r)}\hat{C}C\hat{C}^{-1}J^{(r)}$ is convergent and we know that this is true if all eigenvalues are in the unit disc, namely $|\lambda|<1$ where $\lambda \in spec(A)$. 
   
The above condition is true if the following holds:
\begin{equation}
       0 \leq  \min_i\sum_{l=1,\\l\neq i}^n | a^{(r)}_{il} | \leq \max_i \sum_{l=1,\\l\neq i}^n | a^{(r)}_{il}| \leq 1.
\end{equation}
    
The above implies that:
\begin{equation}\label{c1}
        0\leq\sum_{l = 1, \\l\neq i}^n |a^{(r)}_{il}| < 1.
\end{equation} 
Now we know that $|a^{(r)}_{ij}|=|j^{(r)}_{ii}\hat{c}_{ii}c_{il}\frac{1}{\hat{c}_{ll}}j^{(r)}_{ll}|= |\hat{c}_{ii}c_{il}\frac{1}{\hat{c}_{ll}}| \textcolor{white}{...} \forall i\neq l$. Condition (\ref{c1}) is then equivalent to condition (\ref{eq:suffcondition}) and this concludes the proof. 
$\square$
\end{proof}

To understand condition (\ref{eq:suffcondition}) let us assume that company $i$ owns $90\%$ of company $j$ and has no share in other companies. Assume also that company $j$ in turns does not invest in other companies, then $C_{ii} = 0.1$ and $\sum_{j=1}^n \frac{C_{ij}}{\hat{C}_{jj}}=0.9$. Then  $\hat{C}_{ii} \Big(\sum_{j=1}^n \frac{C_{ij}}{\hat{C}_{jj}}\Big)= 0.09<1$. From this example we can see that when there is a company that owns a large share of another company which in turn does not invest much in other companies, then the former company represents a stable node in the network.

Condition (\ref{cilsub}) implies that  $A^{(r)}_-$ is a null matrix. We can relax this condition and assume $A^{(r)}_-$ being nonnull. We can still prove positive invariance in the case where $X_-(t)$ is bounded. 

Consider the hyperbox $H:=\{X_- \in \mathbb R^{|\mathcal N^{(r)}|}| \, X_l \leq \bar X_l,\, \forall \, l\in \mathcal N^{(r)} \} \subset \mathbb R^{|\mathcal N^{(r)}|}$ where $\bar X_l$ are given upper bounds. Assume that  $X_-(t) \in H$ for all $t\geq 0$. Also denote $\bar X_-:= [\bar X_l]_{l \in \mathcal N^{(r)}}$.
 
\begin{theorem}\label{thm33}
Consider dynamics (\ref{eq:error_model_33}), and assume $X_-(t) \in H$ for all $t \geq 0$. If 
\begin{equation}\label{pos_sub1}
(Dp)_{i} \geq \Big((\hat{C}^{-1}-C\hat{C}^{-1})\underline{v}\Big)_i -\Big(A^{(r)}_- \bar X_-\Big)_i \forall i \in \mathcal P^{(r)} \end{equation} 
then for all $i \in \mathcal P^{(r)}$
\begin{equation}\label{pos_sub11111}
    X_i(0) \geq  0 \Longrightarrow 
    X_i(t) \geq 0, \quad \mbox{for all $t=0,1,2,\ldots$}.
\end{equation}
\end{theorem}
        
\begin{proof}
Let us consider dynamics
\begin{equation}\label{pos_sub1111}X_+(t+1) = A^{(r)}_+ X_+(t) + A^{(r)}_- X_-(t) + b^{(r)}_+(X(t)).
\end{equation}
We know that $A_+^{(r)} \geq 0$. 

Our aim is now to show that 
$$A^{(r)}_- X_-(t) + b^{(r)}_+(X(t)) \geq 0.$$
To do this, we recall that for all $i \in \mathcal P^{(r)}$, $$b^{(r)}_i(X)\geq -\Big(A^{(r)}_- \bar X_-\Big)_i$$ is true if and only if
\begin{equation}
        \begin{array}{r}
        \Big((\hat{C}C\hat{C}^{-1}-I)\underline{v} + \hat{C} \Big(Dp - U(J^{(r)}X) - \mathbf{1}\frac{\beta}{2}\Big) \Big)_i
        \geq -\Big(A^{(r)}_- \bar X_-\Big)_i.
        \end{array}
\end{equation}
The above inequality is true if and only if 
\begin{align}
        \begin{array}{cc}
             &\Big(\hat{C}^{-1} (\hat{C}C\hat{C}^{-1}-I)\underline{v} + \hat{C}^{-1} \hat{C}\Big(Dp - U(J^{(r)}X) - \mathbf{1}\frac{\beta}{2} \Big) \Big)_i  \\
             &=\Big( (C\hat{C}^{-1}-\hat{C}^{-1})\underline{v} + Dp - U(J^{(r)}X) - \mathbf{1}\frac{\beta}{2} \Big)_i
             \geq -\hat{C}^{-1}_{ii} \Big(A^{(r)}_- \bar X_-\Big)_i.
        \end{array}    
\end{align}
Now, for all  $i \in \mathcal P^{(r)}$ then we have $U(J^{(r)}X) + \mathbf{1}\frac{\beta}{2}=0$ and for the above we obtain  
\begin{equation}\label{eq:dpbr3}
         (Dp)_{i} \geq \Big(-(C\hat{C}^{-1}-\hat{C}^{-1})\underline{v}\Big)_i-\hat{C}^{-1}_{ii}\Big(A^{(r)}_- \bar X_-\Big)_i,
\end{equation}
which corresponds to  (\ref{pos_sub1}). 
     

To complete the proof it is left to notice that condition (\ref{pos_sub1}) also implies \begin{equation}
\label{pos_sub11}
b^{(r)}_+(X(t)) \geq -A^{(r)}_- \bar X_-. 
\end{equation} 
From (\ref{pos_sub11}) we also have $$A^{(r)}_- X_-(t) + b^{(r)}_+(X(t)) \geq 0,$$
and therefore dynamics (\ref{pos_sub1111})
satisfies the positive invariance condition 
\begin{equation}\label{pos_sub111}
    X_+(0) \geq  0 \Longrightarrow 
    X_+(t) \geq 0, \quad \mbox{for all $t=0,1,2,\ldots$}.
\end{equation}
The above condition coincides with (\ref{pos_sub11111}). This concludes our proof. 
$\square$
\end{proof}

\subsection{Large-scale network.}

Theorem \ref{thm33} has a nice counterpart result in the case of a large-scale network. For the sake of simplicity, assume homogeneity in the shares, namely for a given $0<\alpha<1$,  
\begin{equation}\nonumber
a_{ij}=\left\{
\begin{array}{lll}
\alpha & \mbox{with prob. $\tilde \Theta_i \in [0,1]$}\\ 
0 & \mbox{with prob. $1-\tilde \Theta_i$}
\end{array}
\right., \quad \forall i\in \mathcal P^{(r)}, \, j\in \mathcal N^{(r)}.
\end{equation}
Then we have
$$A^{(r)}_-:=[a_{ij}^{(r)}\in\{0,\alpha\}]_{i\in \mathcal P^{(r)}, \, j\in \mathcal N^{(r)}} \in \{0,\alpha\}^{ |\mathcal P^{(r)}| \times |\mathcal N^{(r)}|}.$$

Note that $\tilde \Theta_i  n$ is the average connectivity (degree) of the network. 

Also assume that for all $j \in \mathcal N^{(r)}$  $\bar X_j=\bar \xi$. Then we can obtain the maximum number of links (cross-sharing) to failed company for which company $i$ is still safe, and this value is given by
$$\hat k_i =\Big \lfloor \frac{b_i^{(r)}(X)}{\alpha \bar \xi} \Big\rfloor.$$
It is worth mentioning that the above condition is obtained by taking into account a worst case analysis, since it derives from (\ref{pos_sub1}) once we replace $\Big(A^{(r)}_- \bar X_-\Big)_i=\tilde n \alpha \bar \xi$ and check for the maximum $\tilde n$ for which (\ref{pos_sub1}) is satisfied. 

The probability of a link to a failed node for each company $i \in \mathcal P^{(r)}$ is
$$\Theta = \frac{|\mathcal N^{(r)}|}{|\mathcal P^{(r)}|+|\mathcal N^{(r)}|}= \frac{|\mathcal N^{(r)}|}{n}.$$

Let $\Delta_i$ be the degree of company $i$, which is given by $$\Delta_i = \tilde \Theta_i n.$$ In other words, for large $n$ the probability of creating a link is equal to the number of links divided by the number of companies, i.e., $\tilde \Theta_i =\frac{\Delta_i}{n}$. Let $\binom{n}{k} = \frac{n!}{k!(n-k)!}$. Then we have
$$\mbox{Pr \Big\{$\Big(A^{(r)}_- \bar X_-\Big)_i\leq \hat k_i \alpha \bar \xi$\Big\} }= \sum_{j=0}^{\hat k_i} \binom{\Delta_i}{j} \Theta^j_i (1-\Theta_i)^{\Delta_i-j}.$$
The above is obtained from the fact that the probability of having $j$ links to failed companies follows the binomial distribution 
$$\mbox{Pr \Big\{$\Big(A^{(r)}_- \bar X_-\Big)_i = j \alpha \bar \xi$\Big\} }=  \binom{\Delta_i}{j} \Theta^j (1-\Theta)^{\Delta_i-j}.$$
The probability of having at most $\hat k$ links to failed companies is then obtained by calculating the cumulative distribution function of the latter. 

Define function 
\begin{equation}\label{hatF}
\hat F_n(  \tau) = \sum_{i=1}^n  \mathbf 1_{\{\hat k_i < \tilde \Theta_i   \tau\}}, \quad \forall \tau=0,\ldots,n
\end{equation}
where $\mathbf 1_{\{\hat k_i < \tilde \Theta_i   \tau\}}$ is the indicator of event $\hat k_i < \tilde \Theta_i   \tau$, namely 

\begin{equation}
\mathbf 1_{\{\hat k_i < \tilde \Theta_i   \tau\}}=
\left\{\begin{array}{ll}
    1 & \mbox{if $\hat k_i < \tilde \Theta_i   \tau$},\\
    0 & \mbox{otherwise.}
    \end{array}\right.
\end{equation}  

For a population of failed companies of size $ \tau$, $\tilde \Theta_i  \tau$ is the expected number of links that a single company $i$ has to failed neighbour companies. Then $\hat F_n( \tau)$ is the number of companies characterized by a safety threshold $\hat k$ which is smaller than the expected number of links to failed companies $\tilde \Theta_i  \tau$.  
We can obtain the estimate of an upper bound of the number of failed companies as follows:
\begin{equation}\label{est}
\widehat{|\mathcal N^{(r)}|}:=\max_{ \tau=0,\ldots,n} \{\hat F_n(  \tau) \geq  \tau\}.
\end{equation}

It is apparent that $\widehat{|\mathcal N^{(r)}|}$ grows with increasing average connectivity $\tilde \Theta n$.

\section{Control Design.}

In this section, we build on previous results on conditions for system~(\ref{eq:error_model_33})  to be positive and stable, to design an external input to guarantee such conditions. The external input can be viewed as a discrete-time control signal and can be interpreted as an investment strategy on external assets. Namely, we design a feasible investment on external assets that guarantee the system to be positive and stable, which means that cascading failures are prevented.  

\subsection{Feedback-feedforward structure.}
\label{sec4.1:notation and background}

We recall that the physical interpretation of the control is in terms of the investments in external assets $Dp$. The correct investments in external assets ensures that the market values of the non-failing companies remain above the threshold market values. We partition control $Dp$ into two parts: 

\begin{itemize}
    \item [\textit{i)}] A constant feedforward part $u^{(1)}$ to ensure that the input vector $b^{(r)}$ is non-negative,
    \item [\textit{ii)}] A dynamical feedback part $u^{(2)}$ to ensure that state space matrix $A^{(r)}$ is positive and satisfies the sufficient condition for stability $0\leq\sum_{l = 1}^n a_{il}^{(r)} < 1$.
\end{itemize}
Then, we have
\begin{align}\label{eq:Dpstart}
     Dp = u^{(1)}&+u^{(2)}, \\ 
     \text{Where:} \textcolor{white}{...}u^{(1)} &= \underline{u}^{(r)}, \\ 
     u^{(2)} &= K^{(r)}X(t).\label{eq:Dpstart3}
\end{align}

It is important to mention that the design of the control signal in terms of investments in external assets is performed at each time period $t$. From (\ref{eq:error_model_33}), replacing the term on external investment $Dp$ by the  feedback-feedforward input as in (\ref{eq:Dpstart})-(\ref{eq:Dpstart3}) yields the following closed-loop dynamical system: 
\begin{equation}\label{eq:tilde_error_control1}
        \left\{ \begin{array}{ll} X(t+1) =(A^{(r)}+J^{(r)}\hat{C}K^{(r)})X(t)  + J^{(r)}((\hat{C}C\hat{C}^{-1}-I)\underline{v} + \hat{C}(\underline{u}^{(r)} - U(J^{(r)}X) - \mathbf{1}\frac{\beta}{2}),\\
        v(t) = J^{(r)}X(t)+\underline{v}.
    \end{array}\right.
\end{equation}

Let $K^{(r)} := J^{(r)}\Tilde{K}^{(r)}$, then the feedback matrix is  $J^{(r)}\hat{C}J^{(r)}\Tilde{K}^{(r)}=J^{(r)}\hat{C}K^{(r)}$ and the closed-loop matrix
$\Tilde{A}^{(r)}$ is:
\begin{equation}\label{ta}
    \begin{array}{lll}
    \tilde{A}^{(r)}  := [\tilde a_{ij}]_{i,j=1,\ldots,n}
     = (A^{(r)}+\hat{C}\Tilde{K}^{(r)})
     =\begin{bmatrix}\hat{c}_{11}\Tilde{k}_{11} & \ldots & a_{1n}+\hat{c}_{11}\Tilde{k}_{1n}\\\vdots & \ddots & \vdots\\
    a_{n1}+\hat{c}_{nn}\Tilde{k}_{n1} &  \ldots & \hat{c}_{nn}\Tilde{k}_{nn}\\\end{bmatrix}.
\end{array}
\end{equation}

Let us also introduce the vector 
\begin{equation}
    \label{tb}
\tilde{b}^{(r)}:=J^{(r)}\Big((\hat{C}C\hat{C}^{-1}-I)\underline{v} + \hat{C}(\underline{u}^{(r)} - U(J^{(r)}X) - \mathbf{1}\frac{\beta}{2})\Big).\end{equation}

Note that $\tilde{b}^{(r)}$ in (\ref{tb}) differs from ${b}^{(r)}$ in (\ref{br}) as the external assets' investment term $Dp$ is now replaced by the yet to be designed control input $\underline{u}^{(r)}$.
System (\ref{eq:tilde_error_control1}) can be rewritten in compact way as follows:
\begin{equation}
\label{manuscriptmodel}
    \left\{ \begin{array}{ll} X(t+1) = \tilde{A}^{(r)}X(t)+\tilde{b}^{(r)},\\
 v(t)= J^{(r)}X(t)+\underline{v}.\end{array} \right.
\end{equation}

In the following of this manuscript, system (\ref{manuscriptmodel}) is referred to as the closed-loop dynamical system and all results will relate to it. To visualize the dynamics of the closed-loop dynamical system, a block diagram of the model is represented in \ref{fig:feed}.

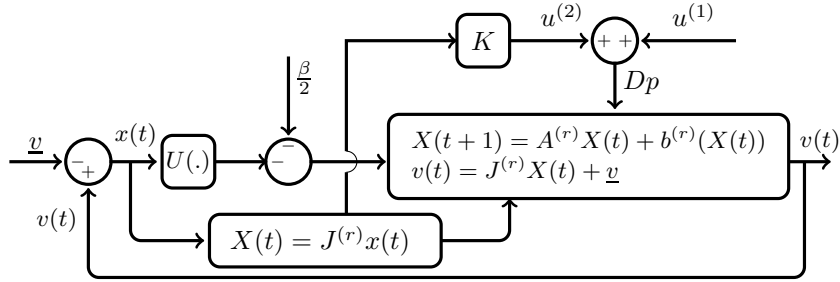
\begin{figure} [h]
\centering
\begin{tikzpicture}[scale=0.70]
\begin{scope}[shift={(0,0)},scale=1]
 \draw [rounded corners, very thick](-0.3,0) rectangle (7.3,1.6);
 \node at (3.5,.8) {\footnotesize 
 \(\displaystyle 
 \begin{array}{ll}
 X(t+1) = A^{(r)}X(t)+b^{(r)}(X(t))\\
 v(t)= J^{(r)}X(t) + \underline{v}
 \end{array}
 \)
 };
 \draw[->, very thick](7.3,.7)--(8.1,.7);\node at (7.9,1.1) {\small $v(t)$};
 \draw[->,rounded corners, very thick] (7.6,.7)--(7.6,-1.5)--(-6,-1.5)--(-6,.2);\node at (-6.6,-0.4) {\small $v(t)$};
 \draw[->,rounded corners, very thick] (-5.2,.7)--(-5.2,-0.8)--
(-3.8,-0.8);
\draw[->,rounded corners, very thick]
(.7,-0.8) --(2,-0.8)--(2,0);\node at (-1.6,-0.8) {$X(t) = J^{(r)}x(t)$};
 \draw[black,rounded corners, very thick ] (-3.7,-1.3) rectangle (0.7,-0.3);
\draw[very thick ] (-6,.7) circle (.43cm);\node at (-6.0,0.5) {\tiny $+$};\node at (-6.2,0.7) {\tiny $-$};
\draw[->, very thick](-7.5,.7)--(-6.5,.7);\node at (-7.0,1.0) {$\underline v $};
\draw[->, very thick ](-5.6,.7)--(-4.7,.7);\node at (-5.1,1.2) {\footnotesize $x(t)$};
\draw[black,rounded corners, very thick ] (-4.6,0.2) rectangle (-3.6,1.2);\node at (-4.1,.7) {\footnotesize $U(.)$};
\draw[->, very thick ](-3.6,.7)--(-2.6,.7);
\draw [very thick ](-2.2,.7) circle (.43cm);\node at (-2.4,0.7) {\tiny $-$};\node at (-2.2,0.99) {\tiny $-$};
\draw [->, very thick ] (-2.2,2.7)-- (-2.2,1.2);\node at (-1.9,2.3) {$\frac{\beta}{2}$};
\draw [->, very thick ] (-1.8,.7)-- (-0.35,.7);
\draw [->, very thick] (4.0,2.6)-- (4.0,1.7);\node at (4.5,2.2) {$Dp$};
\draw [->, very thick] (-1.1,0.9)-- (-1.1,3.0)--(1.0,3.0);
\draw [ very thick] (-1.1,-0.3)-- (-1.1,0.5);
\draw (-1.1,0.5) arc
    [
        start angle=-90,
        end angle=90,
        x radius=.2cm,
        y radius =.2cm
    ] ;
 
\draw[black,rounded corners, very thick ] (1.0,2.5) rectangle (2.0,3.5);\node at (1.5,3) {$K$};
\draw [->, very thick] (2,3.0)--(3.5,3.0);
\draw[very thick ] (4,3.0) circle (.43cm);\node at (3.8,3.0) {\tiny $+$};\node at (4.2,3) {\tiny $+$};
\draw [<-, very thick] (4.5,3.0)--(6.3,3.0);
\node at (3.0,3.5) {$u^{(2)}$};
\node at (5.5,3.5) {$u^{(1)}$};
\end{scope}
\end{tikzpicture}
\vspace{5mm}
\caption{Feedback-feedforward control on assets' investments.} 
\label{fig:feed}
\end{figure}

\subsection{Ensuring $\tilde{b}^{(r)}$ is non-negative.}
\label{sec5.2.1}

In this section, the constant feedforward control $u^{(1)}=\underline{u}^{(r)}$ is designed to ensure that $\tilde{b}^{(r)}$ is non-negative.
\begin{theorem}
\label{trm:u1}
Consider system (\ref{manuscriptmodel}), and let us design     \begin{equation}\label{eq:u1}
\underline{u}^{(r)} = (\hat{C}^{-1}- C\hat{C}^{-1})\underline{v}+ U(J^{(r)}X) + \mathbf{1}\frac{\beta}{2}+J^{(r)} \xi, \quad \xi>0, 
\end{equation}
then vector $\tilde{b}^{(r)}$ is non-negative.
\end{theorem}

\begin{proof}
From Lemma~\ref{trm:bk} it is known that the following constraints should hold for $\tilde{b}^{(r)}$ to be non-negative:
\begin{eqnarray}
\label{constraintb>0_1u}
             \underline u_{i}^{(r)} \leq \Big((\hat{C}^{-1}-C\hat{C}^{-1})\underline{v}\Big)_i +\beta \qquad \text{ if } i \in \mathcal N^{(r)},
                \\
    \label{constraintb>0_1bisu}
                \underline u_{i}^{(r)} \geq \Big((\hat{C}^{-1}-C\hat{C}^{-1})\underline{v}\Big)_i  \qquad \text{if }  i \in \mathcal P^{(r)}.
\end{eqnarray}
   
Substituting (\ref{eq:u1}) in  (\ref{constraintb>0_1u}) we have that for all $i \in \mathcal N^{(r)}$: $$\Big((\hat{C}^{-1}-C\hat{C}^{-1})\underline{v}\Big)_i +\beta +j_{ii}^{(r)} \xi_i \leq \Big((\hat{C}^{-1}-C\hat{C}^{-1})\underline{v}\Big)_i +\beta,$$
which is satisfied as $j_{ii}^{(r)} \xi_i<0$.    

Analogously, substituting (\ref{eq:u1}) in  (\ref{constraintb>0_1bisu}) we have that for all $i \in \mathcal P^{(r)}$: $$\Big((\hat{C}^{-1}-C\hat{C}^{-1})\underline{v}\Big)_i  +j_{ii}^{(r)} \xi_i \geq \Big((\hat{C}^{-1}-C\hat{C}^{-1})\underline{v}\Big)_i,$$
which is satisfied as $j_{ii}^{(r)} \xi_i>0$.
$\square$
\end{proof}
  
The above designed control $u^{(1)}$ once substituted in system~(\ref{manuscriptmodel}) yields the following dynamics:
\begin{equation}\label{eq:tilde_error_controlu12}
        \left\{ \begin{array}{ll} X(t+1) =(A^{(r)}+J^{(r)}\hat{C}K^{(r)})X(t) + J^{(r)}\hat{C} J^{(r)}\xi,\\
        v(t) = J^{(r)}X(t)+\underline{v},
    \end{array}\right.
\end{equation} 
where $\tilde{b}^{(r)}=\hat{C} J^{(r)}x^{(r)}\geq0$. 

\begin{theorem}
\label{trm:u1_pos}
Consider system (\ref{manuscriptmodel}), and assume $X_-(t) \in H$ for all $t \geq 0$. Let us design   $\forall i \in \mathcal P^{(r)}$   \begin{equation}\label{eq:u1_pos}
\underline{u}^{(r)}_i =
\Big((\hat{C}^{-1}-C\hat{C}^{-1})\underline{v}\Big)_i -\hat{C}^{-1}_{ii}\Big(A^{(r)}_- \bar X_-\Big)_i 
+\xi_i, \quad \xi_i>0, 
\end{equation}
then for all $i \in \mathcal P^{(r)}$
\begin{equation}\label{pos_sub11111des}
    X_i(0) \geq  0 \Longrightarrow 
    X_i(t) \geq 0, \quad \mbox{for all $t=0,1,2,\ldots$}.
    \end{equation}
\end{theorem}

\begin{proof}
From Theorem \ref{thm33}, $\forall i \in \mathcal P^{(r)}$ we have that $\tilde{b}^{(r)}_i$ is non-negative if:
       \begin{eqnarray}
    \label{constraintb>0_1bisu_pos}
                \underline u_{i}^{(r)} \geq \Big((\hat{C}^{-1}-C\hat{C}^{-1})\underline{v}\Big)_i -\hat{C}^{-1}_{ii}\Big(A^{(r)}_- \bar X_-\Big)_i.
\end{eqnarray}
Substituting (\ref{eq:u1_pos}) in  (\ref{constraintb>0_1bisu_pos}) we have that for all $i \in \mathcal P^{(r)}$: \begin{equation}
\begin{array}{l}
\Big((\hat{C}^{-1}-C\hat{C}^{-1})\underline{v}\Big)_i -\hat{C}^{-1}_{ii}\Big(A^{(r)}_- \bar X_-\Big)_i 
+\xi_i 
\geq \Big((\hat{C}^{-1}-C\hat{C}^{-1})\underline{v}\Big)_i -\hat{C}^{-1}_{ii}\Big(A^{(r)}_- \bar X_-\Big)_i,
\end{array}
\end{equation}
which is satisfied as $\xi_i>0$.
$\square$
\end{proof}

\subsection{Ensuring $\tilde{A}^{(r)}$ positive and stable.}
\label{sec5.2.2}

In this section, the dynamical feedback control $u^{(2)}$ is designed, to ensure that $\tilde{A}^{(r)}$ is positive and guarantees that the sufficient condition for stability in Theorem \ref{trm:analysis_equilirbium} holds.  As previously stated, we assume that $K^{(r)}:=J^{(r)}\Tilde{K}^{(r)}$. Matrix $\Tilde{K}^{(r)}$ is hereafter referred to as the control input matrix. To find the entries of the control input matrix, the following linear program is developed.

Let us introduce the following new notation 
\begin{equation}\nonumber
    \Tilde{k} := vec(\Tilde{K}) =[\Tilde{K}_{\bullet 1}^T \ldots \Tilde{K}_{\bullet n}^T]^T, \quad \delta := \Big[\frac{a_{\bullet 1}}{\hat{c}_{11}}^T \ldots \frac{a_{\bullet n}}{\hat{c}_{nn}}^T \Big]^T. 
\end{equation}
Note that  $\Tilde{k}$ is the vectorization of $\Tilde{K}$ obtained by stacking the columns of the matrix $\Tilde{K}$ on top of one another. Also, let $ e_i$ be the $i$-th canonical basis vector for the $n$-dimensional space, i.e.,  $ e_i:=[0\ldots 0 \, 1 \, 0 \ldots 0]^T$ and let us denote  
$$\mathbf e_i := e_i\otimes \mathbf 1_n=[\mathbf 0^T\ldots \mathbf 0^T \, \mathbf 1^T_n  \, \mathbf 0^T \ldots \mathbf 0^T]^T.$$

Furthermore, for $\epsilon>0$, let us denote 
$$\gamma_i:=\frac{1}{\hat{c}_{ii}}(1 - \displaystyle\sum\limits_{l=1}^{n} a_{il}) -\epsilon,  \quad \mu_i:=\frac{1}{\hat{c}_{ii}}\displaystyle\sum\limits_{l=1}^{n} a_{il}.$$

The linear program is given by
\begin{eqnarray}\label{lp:LPk}
\mbox{(LP1)}~~~~\text{min}  & -\displaystyle \mathbf 1^T_{n^2} \Tilde{k} \\
\label{LP1} \text{s.t.} & -\Tilde{k} \leq  \delta \\
\label{LP2}& \displaystyle \mathbf e_i^T \Tilde{k} \leq \gamma_i  & 
\forall i=1,\ldots,n \\
\label{LP3}& -\mathbf e_i^T \Tilde{k} \leq \mu_i & \forall i=1,\ldots,n\\
\label{LP4}& \Tilde{k}_{il} = 0 &\forall i=l. 
\end{eqnarray}

\begin{theorem}\label{trm:LPk}
Consider system (\ref{manuscriptmodel}), and let $\Tilde{K}^{(r)}$ be designed according to (\ref{lp:LPk})-(\ref{LP4}), then matrix $\tilde{A}^{(r)}$ is positive, and satisfies the sufficient conditions for stability in  (\ref{eq:suffcondition}).
\end{theorem}

\begin{proof}
For the first part, we wish to show that matrix $\tilde{A}^{(r)}$ is a positive matrix, i.e. $\tilde{A}^{(r)}>0$. This is true if all entries  are non-negative, i.e. $\tilde{a}_{il}\geq0 \textcolor{white}{...} \forall i,l=1,...,n$. From (\ref{ta}) we know that the generic off-diagonal entry is $\tilde{a}_{il}=a_{il}+\hat{c}_{ii}\Tilde{k}_{il}$. Constraint  (\ref{LP1}) then implies 
\begin{equation}\nonumber
    \tilde{a}_{il}=a_{il}+\hat{c}_{ii}\Tilde{k}_{il}\geq 0.
\end{equation}
The above can be rewritten as 
$$\Tilde{k}_{il} \leq  \frac{a_{il}}{\hat{c}_{ii}}  \quad  \forall i,l=1,\ldots,n.$$ This is the component-wise version of (\ref{LP1}).

For the second part, we aim at showing that all eigenvalues of $\tilde{A}^{(r)}$ are within the unit disc. It is well-known that all eigenvalues of a matrix $\tilde{A}^{(r)}$ lay in the unit disc if all diagonal values are zero and the inequalities $0\leq \sum_{l=1}^{n} \tilde{a}_{il}<1 \textcolor{white}{...} \forall i=1,...,n$ hold. Constraint (\ref{LP4}) ensures that all diagonal entries are null, i.e. $\tilde{a}_{il}=0 \textcolor{white}{...} \forall i=l$. Furthermore, constraints (\ref{LP2}) and (\ref{LP3})  ensure that each row sum of matrix $\tilde{A}^{(r)}$ is smaller than one and not smaller than 0. These constraints are obtained by setting $\sum_{l=1}^{n}\tilde{a}_{il}<1$ and $\sum_{l=1}^n\tilde{a}_{il}\geq0$. From (\ref{ta}) the latter can be equivalently written as:
\begin{align}\label{constraints1}
    & \sum_{l=1}^{n} a_{il}+\sum_{l=1}^{n}\hat{c}_{ii}\Tilde{k}_{il}<1, \quad \sum_{l=1}^{n} a_{il}+\sum_{l=1}^{n}\hat{c}_{ii}\Tilde{k}_{il}\geq0.
\end{align}

Let us rewrite (\ref{constraints1}) by isolating  $\Tilde{k}_{il}$ to the left-hand-side. This yields the following two sets of inequalities:
\begin{align}\label{constraints2}
    & \sum_{l=1}^{n}\Tilde{k}_{il}< \frac{1 - \sum_{l=1}^{n} a_{il}}{\hat{c}_{ii}}, \quad \sum_{l=1}^{n}\Tilde{k}_{il}\geq \frac{- \sum_{l=1}^{n} a_{il}}{\hat{c}_{ii}},
\end{align}
which correspond to  (\ref{LP2}) and (\ref{LP3}). This concludes the proof. 
$\square$
\end{proof}

\begin{figure} [h]
\centering
\begin{tikzpicture}[scale=0.65]
\begin{scope}[shift={(0,0)},scale=1]
\draw[->, very thick](0,0)--(12,0);\node at (12,-0.6) {\small $k_{is}$};
\draw[->, very thick](0,0)--(0,6);\node at (-0.6,6) {\small $k_{ir}$};
\draw[fill=red](3,2.2)--(3,3.3)--(10,1)--(5.2,1)--(3,2.2);
\draw[dashed, very thick](3,-1)--(3,4.8);\node at (2.99,5.3) {\small $k_{ir}\geq- \frac{a_{ir}}{\hat c_{ii}}$};
\draw[dashed, very thick](-1,1)--(11.5,1);\node at (11.3,1.5) {\small $k_{il}\geq- \frac{a_{il}}{\hat c_{ii}}$};
\draw[dashed, very thick](-.5,4)--(8,-.5);\node at (7,3.3) {\small $\displaystyle\sum\limits_{l=1}^{n} k_{il}\leq \gamma_i$};
\draw[dashed, very thick](-.5,4.5)--(11.5,.5);\node at (4.7,.5) {\small $\displaystyle\sum\limits_{l=1}^{n} k_{il}\geq \mu_i$};
\end{scope}
\end{tikzpicture}
\vspace{5mm}
\caption{Graphical interpretation of the linear program (\ref{LP1})-(\ref{LP4}).} 
\label{fig:poly}
\end{figure}
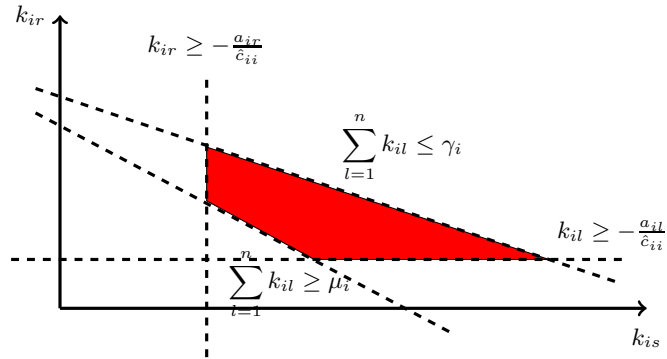

Figure \ref{fig:poly} provides a graphical interpretation of the linear constraints (\ref{LP1})-(\ref{LP4}) projected on a plane $k_{is}-k_{ir}$. The set of feasible solutions is in red. Constraints  (\ref{LP1}) are horizontal and vertical lines (dashed) which correspond to inequalities $k_{is}\geq- \frac{a_{is}}{\hat c_{ii}}$ and $k_{ir}\geq- \frac{a_{ir}}{\hat c_{ii}}$ respectively. The economic interpretation is as follows. The horizontal and vertical lines (dashed) identifies a minimum cash inflow from external investment to compensate the drop in market value $i$ due to companies $s$ and $r$ failing. Constraint (\ref{LP2}) is represented by a diagonal line which corresponds to hyperplane $\sum k_{il}= \gamma_i$ and identifies the feasible halfspace $\sum k_{il} \leq  \gamma_i$ (points below the diagonal line). Similarly, Constraint (\ref{LP3}) is represented by a diagonal line which corresponds to hyperplane $\sum k_{il}= \mu_i$ and identifies the feasible halfspace $\sum k_{il} \geq  \mu_i$ (points above the diagonal line). 

\subsection{Finding the investment choices $D$.}
\label{sec5.2.3}

Once both the constant feedforward and dynamical feedback parts of the control are designed, we need to translate such control into investment choices through an opportune design of matrix $D$. This can be achieved by solving the following linear program, which has the structure of a typical network flow problem.

Let us define the new notation
\begin{equation}\nonumber
\begin{array}{lll}
d:=vec(D)=[D_{\bullet 1}^T \ldots D_{\bullet m}^T]^T,\\
\mathbf {\hat e_i} := e_i\otimes \mathbf 1_m=[\mathbf 0^T\ldots \mathbf 0^T \, \mathbf 1^T_m  \, \mathbf 0^T \ldots \mathbf 0^T]^T,\\
\mathbf {\tilde e_i} := \mathbf 1_m \otimes e_i =[e_i^T\ldots e_i^T]^T,\\
\mathbf {\theta_i} :=[p_1e_i^T\ldots p_m e_i^T]^T.
\end{array}
\end{equation}

The linear program is then given by
\begin{eqnarray}\label{LP:Dp}
\mbox{(LP2)}~~~~\text{min}  & -\displaystyle \mathbf 1_{nm}^T d 
\\\label{LP:Dp1}
\text{s.t.}  & \displaystyle\mathbf {\theta_i}^T d  = u^{(1)}_{i}+u^{(2)}_{i} & \forall i=1,\ldots,n 
\\
 \label{LP:Dp2}& \displaystyle \mathbf {\tilde e_i}^T d \leq  1 & \forall i=1,\ldots,n
\\
 \label{LP:Dp3}& -d \leq 0 \\
 \label{LP:Dp4}& \displaystyle \mathbf {\hat e_l}^T d \leq  1 & \forall l=1,\ldots,m
\end{eqnarray}

The linear equality constraints  (\ref{LP:Dp1}) receive $u^{(1)}$ and $u^{(2)}$ as input and return the entries $d_{il}$, $i,l=1,\ldots, n$. Furthermore, (\ref{LP:Dp1}) works only if the entries of $u^{(1)}_{i}+u^{(2)}_{i}$ are non-negative. However, there exist possibilities that $u^{(1)}_{i}+u^{(2)}_{i}$ has negative entries. To tackle this problem, the different cases of negative entries are discussed below: 
\begin{itemize}
    \item [\textit{i)}] If a company $i$ has failed, its future investments in external assets do not matter anymore. Therefore, for a failed company $i$, entry $u^{(1)}_{i}+u^{(2)}_{i}$ should always be set to zero. 
    \item [\textit{ii)}] If a company $i$ has not failed but its entry $u^{(1)}_{i}+u^{(2)}_{i}$ is negative, this entry should be set to zero as well, because this company requires no further investments to survive.
\end{itemize}
 
Therefore, all negative entries $u^{(1)}_{i}+u^{(2)}_{i}<0$ should be set to zero prior to running the above linear program.

Consider the designed control $D p$, where $D$ denotes the investment choices. Then the linear program (\ref{LP:Dp})-(\ref{LP:Dp4}) finds the investment choices for all companies. To see this, note that constraint (\ref{LP:Dp1}) ensures that the choice of investments in external assets for company $i$ results in the desired market value to survive, as the required investments in external assets for the non-failing companies are defined by $u^{(1)}_{i}+u^{(2)}_{i}$. Constraint (\ref{LP:Dp2}) ensures that a company can invest for no more than 100\% to ensure that companies do not spend more than possible. Constraint (\ref{LP:Dp3}) ensures that only positive investments are made since it is not possible to negatively invest in external assets. Constraint (\ref{LP:Dp4}) ensures that companies can never invest in more than what is left of a certain external asset. Therefore, they have to strategically decide in which asset they want to invest.

\begin{figure} [h]
\centering
\begin{tikzpicture}[scale=0.65]
\begin{scope}[shift={(0,0)},scale=1]
\draw [very thick ](0,0) circle (.4cm);\node at (0,0) {$p_1$};
\node at (-2.5,0.5) {$\sum \limits_{i=1}^{n} p_1 d_{i1} \leq p_1$};
\draw [very thick ](0,-2) circle (.4cm);\node at (0,-2) {$p_2$};
\node at (-2.5,-1.5) {$\sum \limits_{i=1}^{n} p_2 d_{i2} \leq p_2$};
\node at (0,-4) {\large $\vdots$};
\draw [very thick ](0,-6) circle (.4cm);\node at (0,-6) {$p_m$};
\node at (-2.5,-5.5) {$\sum \limits_{i=1}^{n} p_m d_{im} \leq p_m$};
\end{scope}
\begin{scope}[shift={(4,0)},scale=1]
\draw [very thick ](0,0) circle (.4cm);\node at (0,0) { $1$};\node at (2.5,0.5) {$w_1=u^{(1)}_1+u^{(2)}_1$};
\draw [very thick ](0,-2) circle (.4cm);\node at (0,-2) {$2$};\node at (2.5,-1.5) {$w_2=u^{(1)}_2+u^{(2)}_2$};
\node at (0,-4) {\large $\vdots$};
\draw [very thick ](0,-6) circle (.4cm);\node at (0,-6) {$n$};\node at (2.5,-5.5) {$w_n=u^{(1)}_n+u^{(2)}_n$};

\end{scope}
\draw[->, very thick](0.4,0)--(3.6,0);
\draw[->, very thick](0.4,-2)--(3.6,-2);
\draw[->, very thick](0.4,-6)--(3.6,-6);
\draw[->, very thick](0.4,0)--(3.6,-1.8);
\draw[->, very thick](0.4,0)--(3.6,-5.8);
\draw[->, very thick](0.4,-2)--(3.6,-0.2);
\draw[->, very thick](0.4,-2)--(3.4,-5.8);
\draw[->, very thick](0.4,-6)--(3.6,-0.4);
\draw[->, very thick](0.4,-6)--(3.6,-2.2);
\end{tikzpicture}
\vspace{5mm}
\caption{Network flow problem representation of linear program  (\ref{LP:Dp})-(\ref{LP:Dp4}).} 
\label{fig:nfp}
\end{figure}
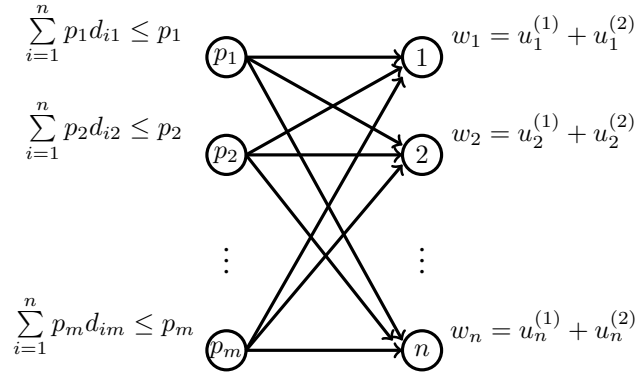

Figure \ref{fig:nfp} provides a network flow problem representation of linear program (\ref{LP:Dp})-(\ref{LP:Dp4}). The network topology is the one of a bipartite graph where the nodes on the left represents assets and the nodes on the right represent companies. We observe capacity constraints of type $\sum \limits_{i=1}^{n} p_l d_{il} \leq p_l$ for each asset $l=1,\ldots,m$. Such constraints are equivalent to (\ref{LP:Dp4}). Similarly we have demand constraints of type $\sum\limits_{l=1}^{m}p_{l}d_{il} =w_i= u^{(1)}_{i}+u^{(2)}_{i}$ for each company $i=1,\ldots,n$. Such constraints are (\ref{LP:Dp1}).

\section{Numerical analysis.}
\label{sec:num}

Let us simulate for a set of one-hundred companies.  All  parameters are listed in Table \ref{tbl:data1}. In particular, we set the initial states within a range between $1$ and $X(0)^+:=\max_{i=1,\ldots,n} X_i(0)=5000$. We set the threshold $\underline{v}=100$ namely $2\%$ of $X(0)^+$. We set the incurred failure cost $\beta=X(0)^+:=5000$. We set the asset values within the range from $1$ to about $0.1$ of $X(0)^+$. Each company invests on a different external asset (namely matrix $D=\mathbb I_{(100 \times 100)}$) and the values of the assets are increasing with step size equal to $6$. By doing this, we obtain a first set of (seventeen) companies  for which  condition~(\ref{constraintb>0_1bis}) is violated and a second set of (eighty-three) companies for which the same condition holds true.  All simulations are carried out with MATLAB on an Intel(R) Core(TM)2 Duo, CPU P8400 at 2.27 GHz and a 3GB of RAM.  The horizon window consists of $T=300$ iterations.

\begin{table}[h!]
\centering
 \begin{tabular}{|c | c|}
 \hline
 \textbf{Parameters} & \textbf{Value} \\ [0.5ex] 
 \hline\hline
 Number of companies & $n=100$\\[0.5ex] \hline
 Initial states & $X(0)=\begin{bmatrix} 1 & 5000  \end{bmatrix}^{50}$\\ [.1ex] \hline
 Threshold market value & $\underline{v}=100$\\[0.5ex] \hline
 Incurred failure costs & $\beta=5000$ \\[0.5ex] \hline
 Assets values& $p=\begin{bmatrix} 1 & 7 & 13 & \ldots &  595 \end{bmatrix}'$\\[.1ex] \hline
 Investments on assets& $D=\mathbb I_{(100 \times 100)}$\\[.1ex] \hline
 \hline
 \end{tabular}
 \caption{Parameters' values}
 \label{tbl:data1} 
\end{table}

We build a  cross-holding matrix as follows:
$C_{ij}=\frac{1}{10n}$ with probability 20\% and zero otherwise. In other words, each company has a share of 0.001\% in  two out of ten other companies. Note that this is the adjacency matrix of a random network with average degree equal to 20. In other words, all companies have approximately the same degree equal to 20. For each company, we simulate the sample error according to dynamics~(\ref{eq:error_model_33}). The pseudo-code of the simulation algorithm is in Table~\ref{fig:algorithm}.
 

\begin{table}[h!]\normalsize
\begin{center}
\begin{tabular}{p{9cm}}\\
\toprule
\textbf{Input:} Set of parameters as in Table \ref{tbl:data1}  \\
\textbf{Output:} Time-evolution of error $X(t)$\\
$\quad 1: $ \textbf{Initialize.} Generate random matrix $C$, \\
$\quad \quad$ initial values $X(0)$, assets' values $p$\\
$\quad 2: $ \textbf{for} time $t=1,\ldots,T$ \textbf{do}\\
$\quad 3: \quad$ Calculate stabilizing control $u^{(1)}$ as in (\ref{eq:u1_pos})\\
$\quad 4: \quad$ Calculate  control $u^{(2)}$ solving LP (\ref{lp:LPk})-(\ref{LP4}) \\
$\quad 5: \quad$ Calculate  error according to dynamics (\ref{eq:error_model_33}); \\
$\quad 6: $ \textbf{end for} \\
$\quad 7:$ Display cluster of nodes violating~(\ref{constraintb>0_1bis}) to the left \\
$\quad \quad $ and all other nodes to the right \\
$\quad 8:$ Color red failed companies and blue all others\\
$\quad 9: $ \textbf{STOP}\\
\bottomrule
\end{tabular}
\end{center}\caption{Simulation algorithm} \label{fig:algorithm}
\end{table}


The outputs of the simulations are visualized in Figs.~\ref{fig:error}-\ref{fig:hatF20}.
In particular, Fig.~\ref{fig:error} displays a slowed down (step size $dt=0.1$) time plot of the evolution of the error $X(t)$ for $t=1,\ldots,T$. Seventeen companies fail  in the first time instances and eventually four more companies fail. So in total we observe twenty-one failed companies at the end of the horizon (stead-state). In Fig.~\ref{fig:propagation}, companies (nodes) for which condition~(\ref{constraintb>0_1bis}) is violated are clusterized to the left and all other nodes to the right. The size of the node increases with the investment on external assets (bigger nodes invest more on external assets).
\begin{figure}[h!]
    \centering
    \includegraphics[width=0.5\textwidth]{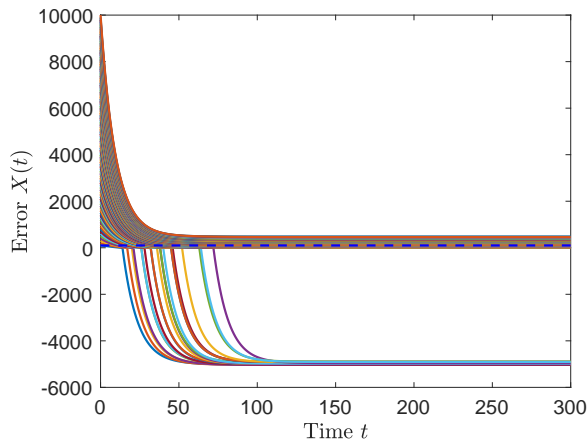} \caption{Time plot of the evolution of the error $X(t)$ obtained from (\ref{eq:error_model_33}). Threshold is indicated with a dashed line.}
    \label{fig:error}
\end{figure}
Simulations are initialized in the orthant where all companies are initially healthy (no failures), see plot at time $t=0$ (top-left). However, as condition ~(\ref{constraintb>0_1bis}) does not hold true for a subset of companies (vector $b^{(r)}(X)$ is not positive), then we have that (\ref{pos}) does not hold true as well, and the trajectory of the error $X(t)$ leaves the orthant. Some companies start failing and failures propagate throughout the network quickly, see plots at time $t=30$ (top-right) and $t=60$ (bottom-left). Failure involves all companies for which ~(\ref{constraintb>0_1bis}) does not hold, but eventually also companies for which~(\ref{constraintb>0_1bis}) holds true. This occurs as condition~(\ref{cilsub}) does not hold, namely we have healthy companies from the right cluster which are still connected with failed companies from the left cluster. In other words, it is not true that at time $t=60$, $c_{il}=0$, for all $i,l \in N$, $i\in \mathcal P^{(r)}$, $l \in \mathcal N^{(r)}$. Therefore also (\ref{pos_sub}) does not hold and three healthy companies eventually fail as well. You can think of $\mathcal P^{(r)}$ at time $t=60$ as the cluster to the right and the fact that three companies in this cluster eventually fail implies that (\ref{pos_sub}) does not hold. In other words,  the set $\mathcal S^{(r)}$ at time $t=60$ is not positively invariant. Note that eventually $20 \%$ of the companies fail, which is in agreement with the fact that each company has on average 20 neighbors and $\frac{1}{5}$ of such neighbors is on average failed, namely each company has on average 4 failed neighbor companies. If we look at the values of $\hat k$ we notice that the first 20 companies have $\hat k \leq 4$. In Fig. \ref{fig:hatF20} we plot function $\hat F_n(\tau)$ as in (\ref{hatF}) (thick line). Estimate (\ref{est}) is the intersection point with bisect (thin line) which amounts to a bit less the 20\% of the total.
\begin{figure}[h!]
    \centering
    \includegraphics[width=0.55\textwidth]{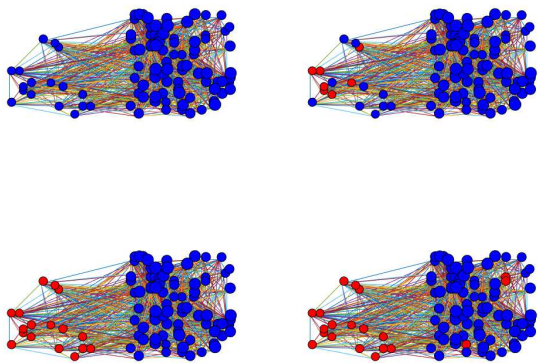}
            \caption{Cascading failure for a random network with average connectivity $\langle \kappa \rangle = 20$: failed companies in red at $t=0,30,60,T$.}
    \label{fig:propagation}
\end{figure}
 
\begin{figure}[h!]
    \centering
    \includegraphics[width=0.5\textwidth]{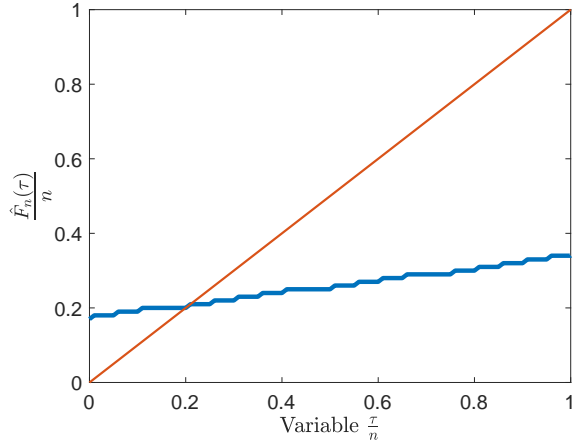}
            \caption{Function $\hat F_n(\tau)$ as in (\ref{hatF}) (thick line). Estimate (\ref{est}) is the intersection point with bisect (thin line). }
    \label{fig:hatF20}
\end{figure}

We repeat the experiment with a dense graph and the result is displayed in Figs.~\ref{fig:error80}-\ref{fig:hatF80}.  Here we increase the probability of cross-holding from 20\% to 80\%. This means that we have a random network with an average degree equal to 80. Each company has cross-shares to about 80\% of the other companies.  In particular, Fig. \ref{fig:error80} displays a slowed down (step size $dt=0.1$) time plot of the evolution of the error $X(t)$ for $t=1,\ldots,T$. Eventually forty-eight companies fail at the end of the horizon (stead-state). This amounts to about $50 \%$ of the companies, which is in agreement with the fact that each company has on average 80 neighbors and $\frac{1}{2}$ of such neighbors is on average failed, namely each company has on average 40 failed neighbor companies. If we look at the values of $\hat k$ we notice that the first 50 companies have $\hat k \leq 40$. In Fig. \ref{fig:hatF80} we plot function $\hat F_n(\tau)$ as in (\ref{hatF}) (thick line). Estimate (\ref{est}) is the intersection point with bisect (thin line) which amounts to about 50\% of the total.

\begin{figure}[h!]
    \centering
    \includegraphics[width=0.5\textwidth]{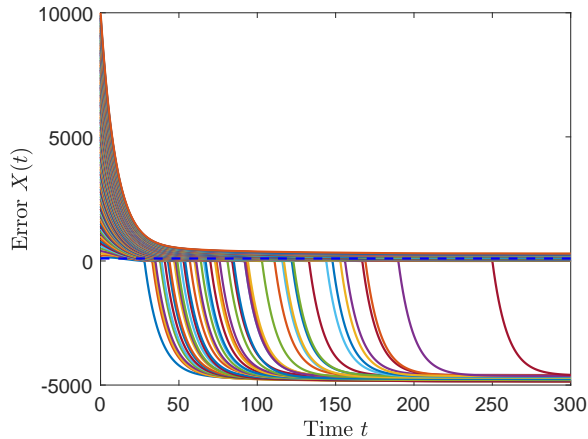} \caption{Time plot of the evolution of the error $X(t)$ obtained from~(\ref{eq:error_model_33}). 
            }
    \label{fig:error80}
\end{figure}

\begin{figure}[h!]
    \centering
    \includegraphics[width=0.55\textwidth]{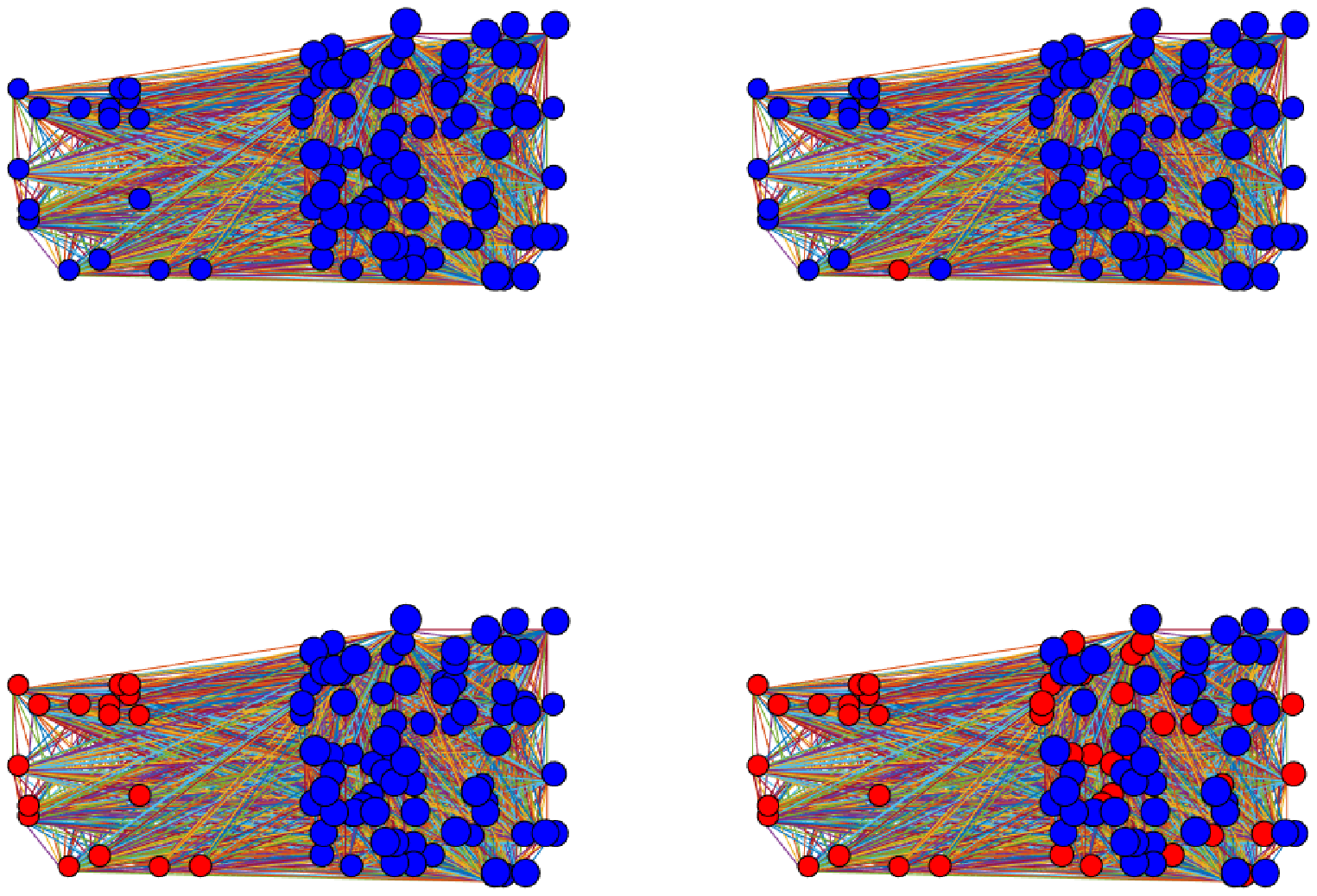}
            \caption{Cascading failure for a random network with average connectivity $\langle \kappa \rangle = 80$: failed companies in red at $t=0,30,60,T$.}
    \label{fig:propagation3}
\end{figure}

\begin{figure}[h!]
    \centering
    \includegraphics[width=0.5\textwidth]{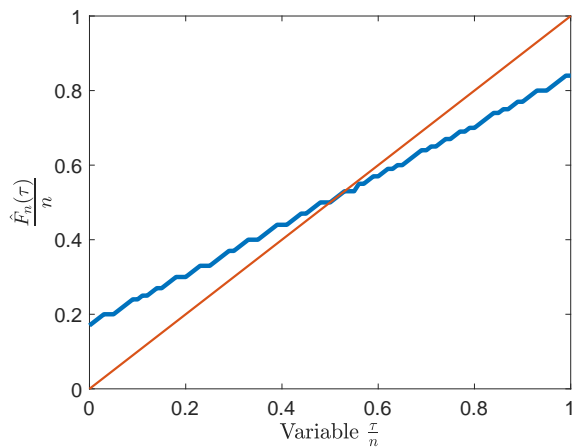}
            \caption{Function $\hat F_n(\tau)$ as in (\ref{hatF}) (thick line). Estimate (\ref{est}) is the intersection point with bisect (thin line). }
    \label{fig:hatF80}
\end{figure}


We repeat the experiment a third time with a sparse graph and the result is displayed in Figs.~\ref{fig:Xpw}-\ref{fig:hatFpw}.  Here the probability of cross-holding follows a power law $P(\kappa) \sim \kappa^{-\rho}$, where $P(\kappa)$ is the fraction of nodes in the network having $\kappa$ out-degree connections to other nodes and $\rho=2.1$. This distribution generates a scale-free networks with very many nodes poorly connected and very few hubs. The average connectivity is $\langle \kappa \rangle=20$. This time the set $\mathcal S^{(r)}$ at time $t=60$ is positively invariant (no healthy companies from the right cluster fail). In particular, Fig. \ref{fig:Xpw} displays a slowed down (step size $dt=0.1$) time plot of the evolution of the error $X(t)$ for $t=1,\ldots,T$. Eventually seventeen companies fail  at the end of the horizon (stead-state). Note that eventually about $20 \%$ of the companies fail, which is in agreement with the fact that each company has on average 20 neighbors and $\frac{1}{5}$ of such neighbors is on average failed, namely each company has on average 4 failed neighbor companies. In Fig. \ref{fig:hatFpw} we plot function $\hat F_n(\tau)$ as in (\ref{hatF}) (thick line). Estimate (\ref{est}) is the intersection point with bisect (thin line) which amounts to a bit less the 20\% of the total.

\begin{figure}[h!]
    \centering
    \includegraphics[width=0.5\textwidth]{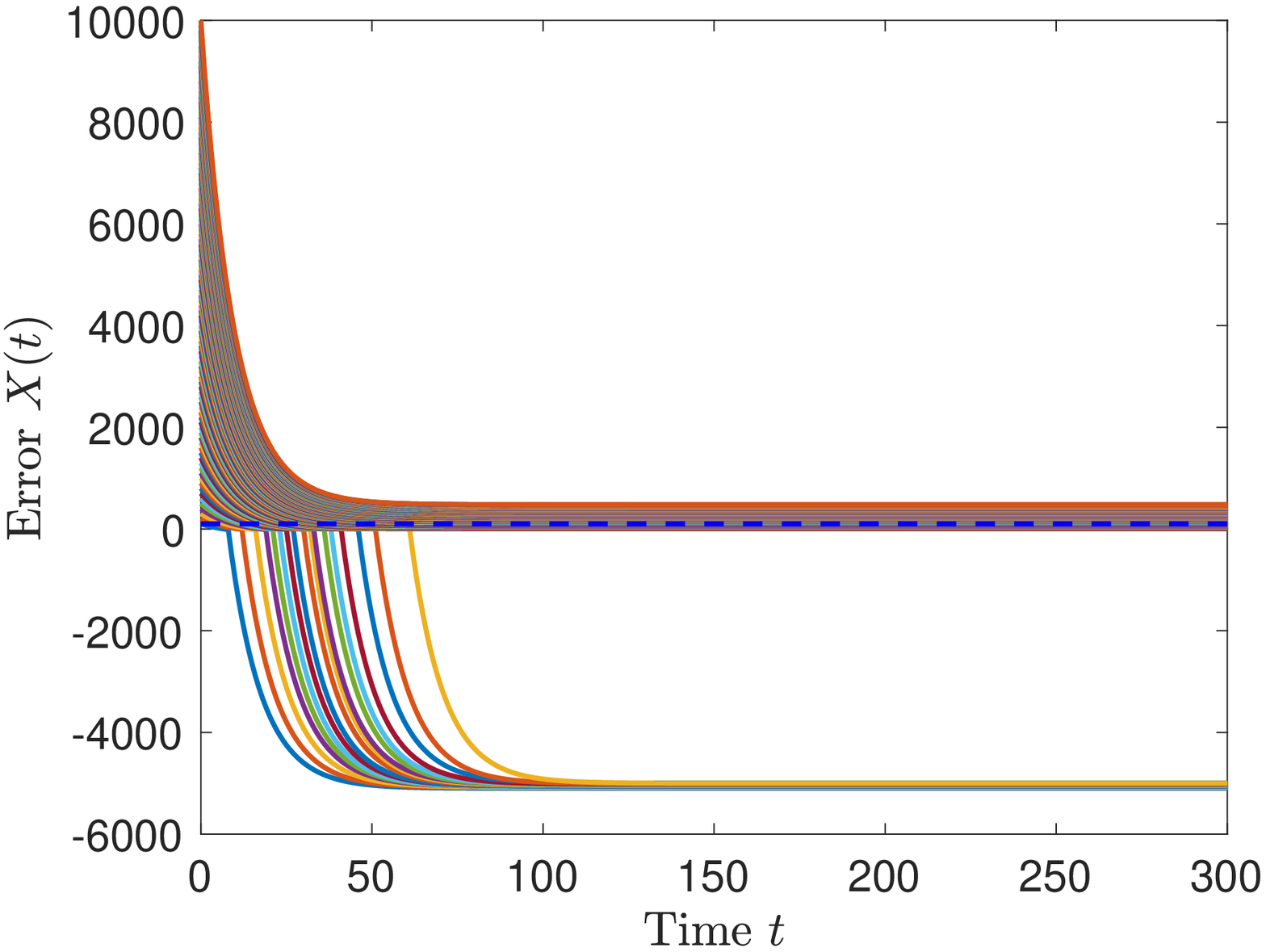} \caption{Time plot of the evolution of the error $X(t)$ obtained from~(\ref{eq:error_model_33}). 
            }
    \label{fig:Xpw}
\end{figure}

\begin{figure}[h!]
    \centering
    \includegraphics[width=0.55\textwidth]{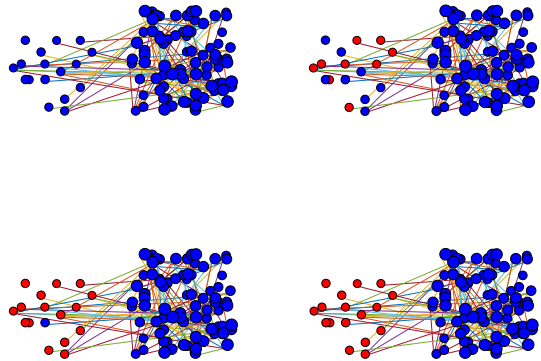}
            \caption{Cascading failure for a scale-free network with  average connectivity $\langle \kappa \rangle=2$: failed companies in red at $t=0,30,60,T$.}
    \label{fig:propagation4}
\end{figure}

\begin{figure}[h!]
    \centering
    \includegraphics[width=0.5\textwidth]{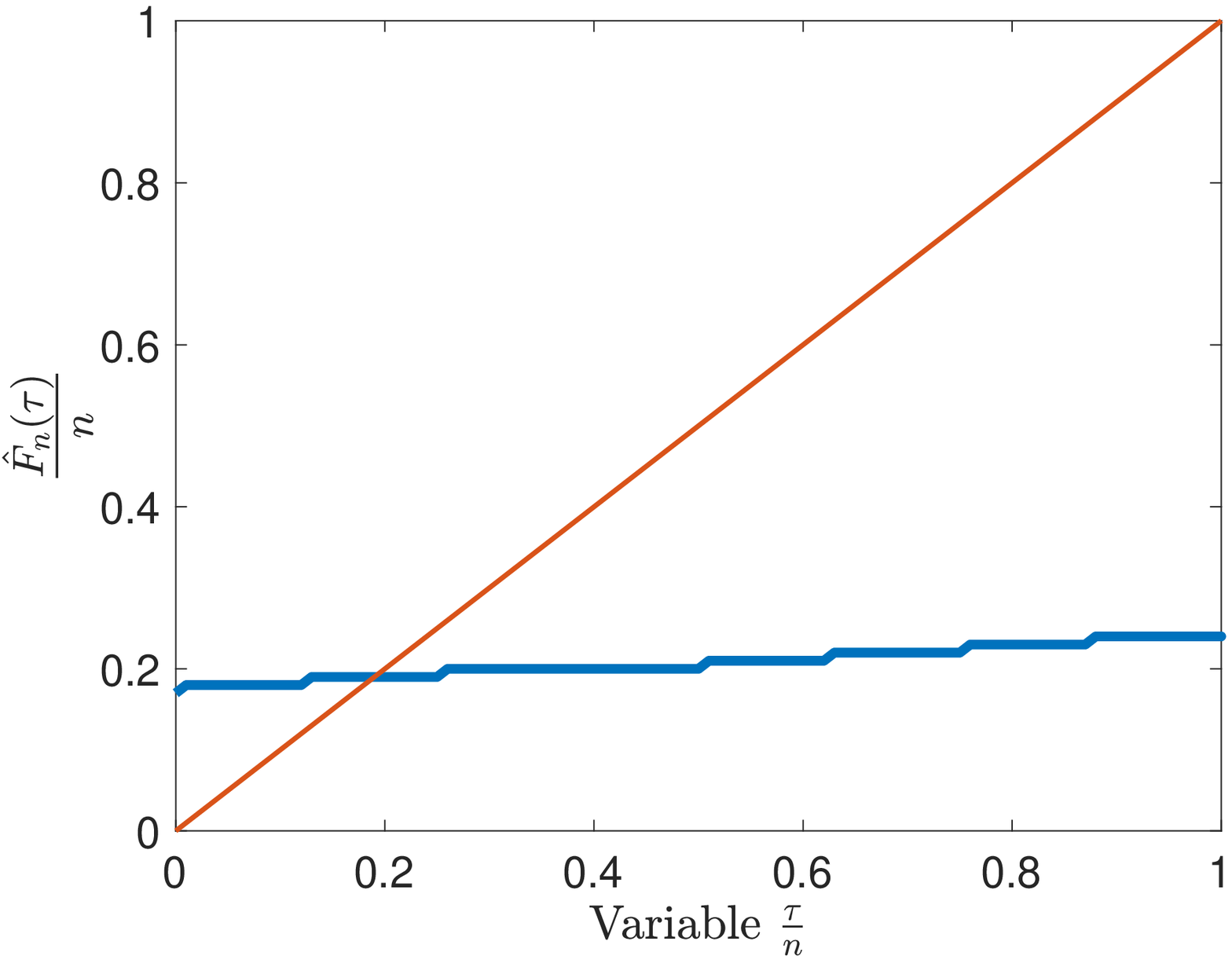}
            \caption{Function $\hat F_n(\tau)$ as in (\ref{hatF}) (thick line). Estimate (\ref{est}) is the intersection point with bisect (thin line). }
    \label{fig:hatFpw}
\end{figure}

In a fourth set of simulations we consider again a random network with an average degree $\langle \kappa \rangle=80$. Figure~\ref{fig:u1} displays the time plot of the error (left). Failures propagate through the whole network.  
Figure~\ref{fig:u1} (right) shows the effect of the stabilizing control $u^{(1)}$ as in (\ref{eq:u1_pos}). The control becomes active at time $t=60$ thus preventing further cascade of failures. The y-axis on the right shows the magnitude of $u^{(1)}$, which can be viewed as cash flow deriving from external assets. 

\begin{figure}[h!]
\begin{tabular}{ll}
\includegraphics[scale=0.4]{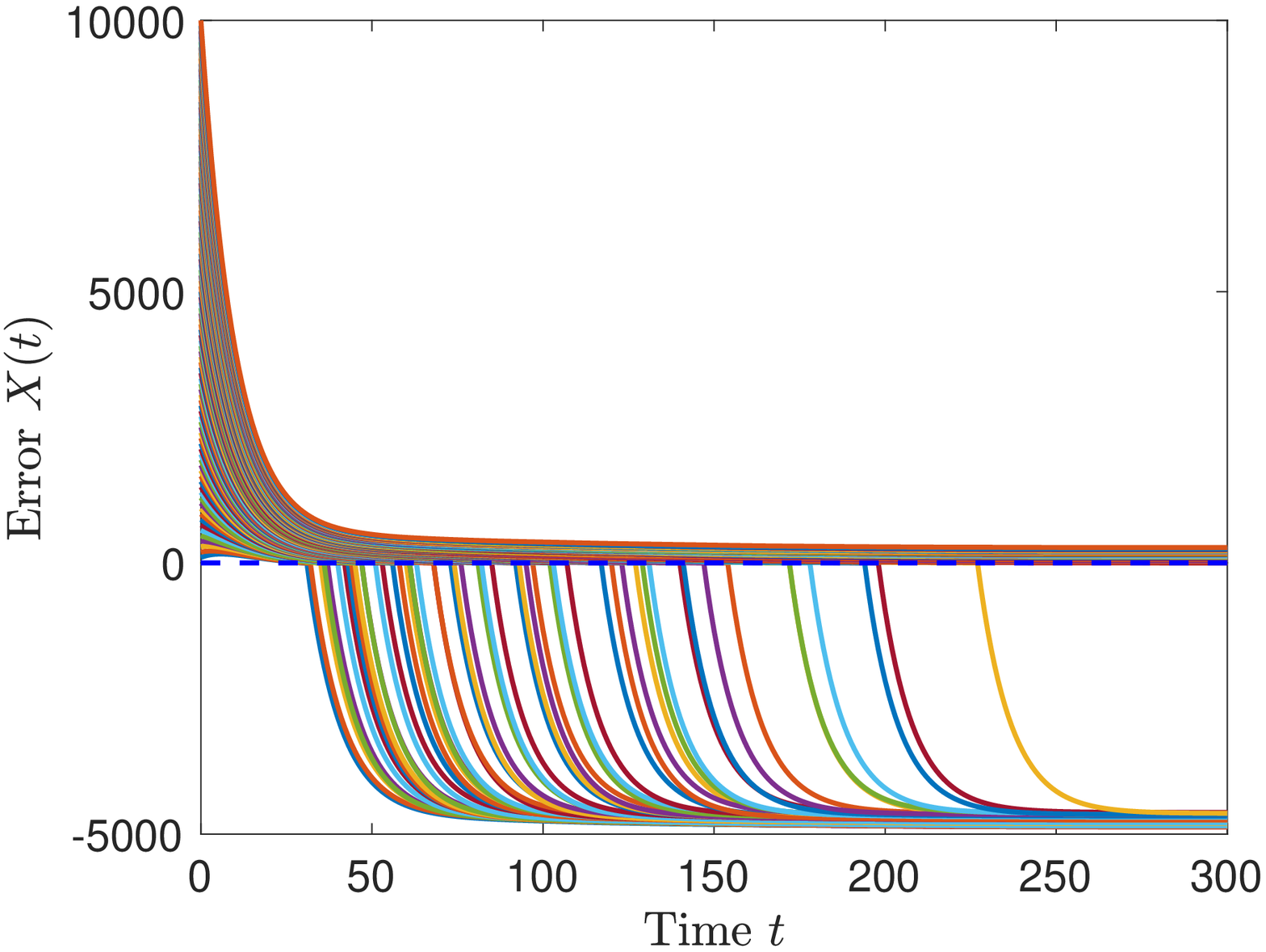}
\hspace{-.6cm}
\includegraphics[scale=0.4]{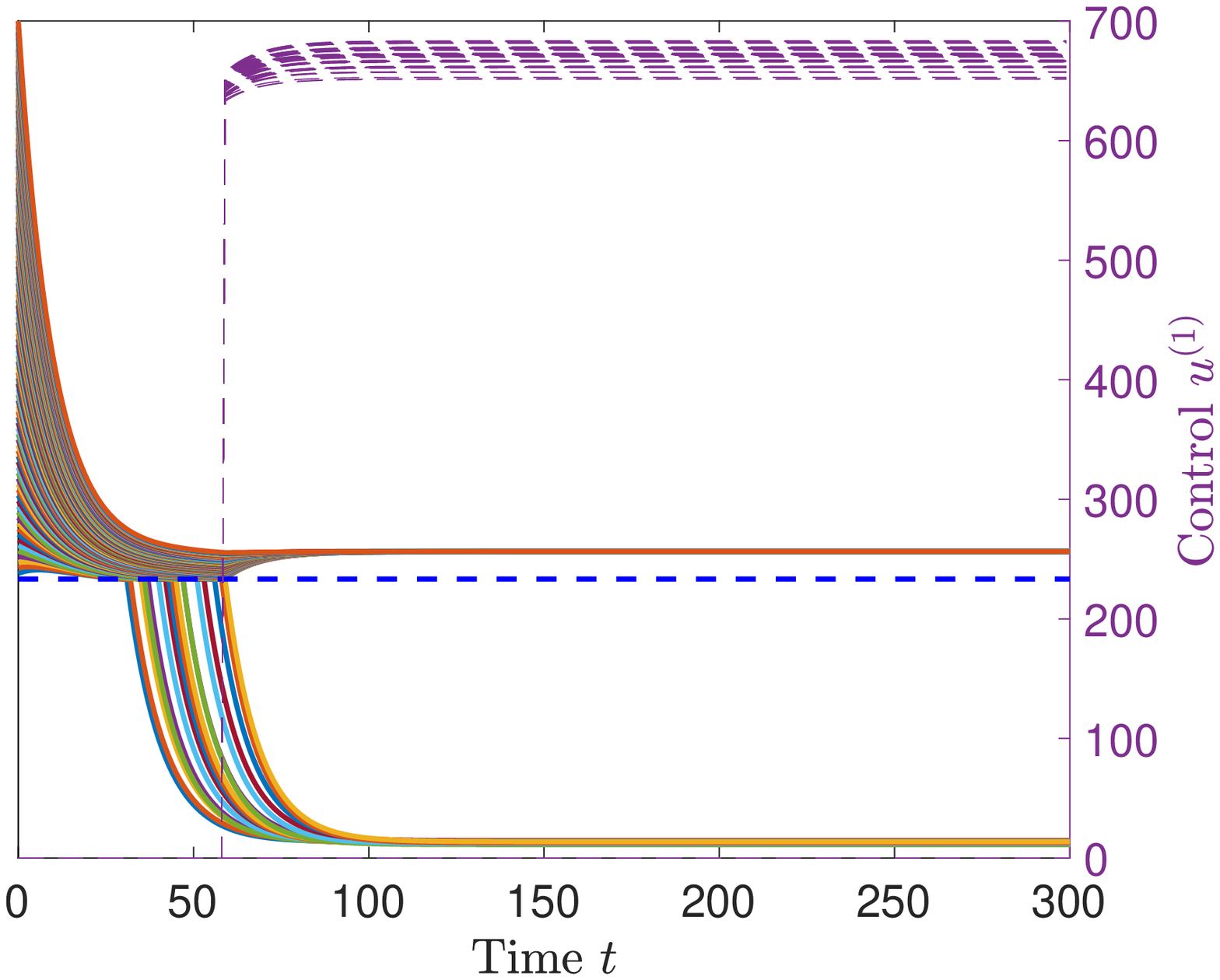}
\end{tabular}
\caption{Left: Failures propagate through the whole network. 
Right: Stabilizing control $u^{(1)}$ (\ref{eq:u1_pos}) activates at time $t=60$.}
\label{fig:u1}
\end{figure}

In a fifth set of simulations we consider again a random network with an average degree $\langle \kappa \rangle\approx 20$. Figure~\ref{fig:u2} displays the time plot of the error (left). Dynamics is very slow and close to instability and indeed we obtain eigenvalues very close to 1 in absolute value. Figure~\ref{fig:u2} (right) shows the effect of the stabilizing control $u^{(2)}$ where $\hat{K}^{(r)}$ is designed according to the linear program (\ref{lp:LPk})-(\ref{LP4}).

\begin{figure}[h!]
\begin{tabular}{ll}
\includegraphics[scale=0.4]{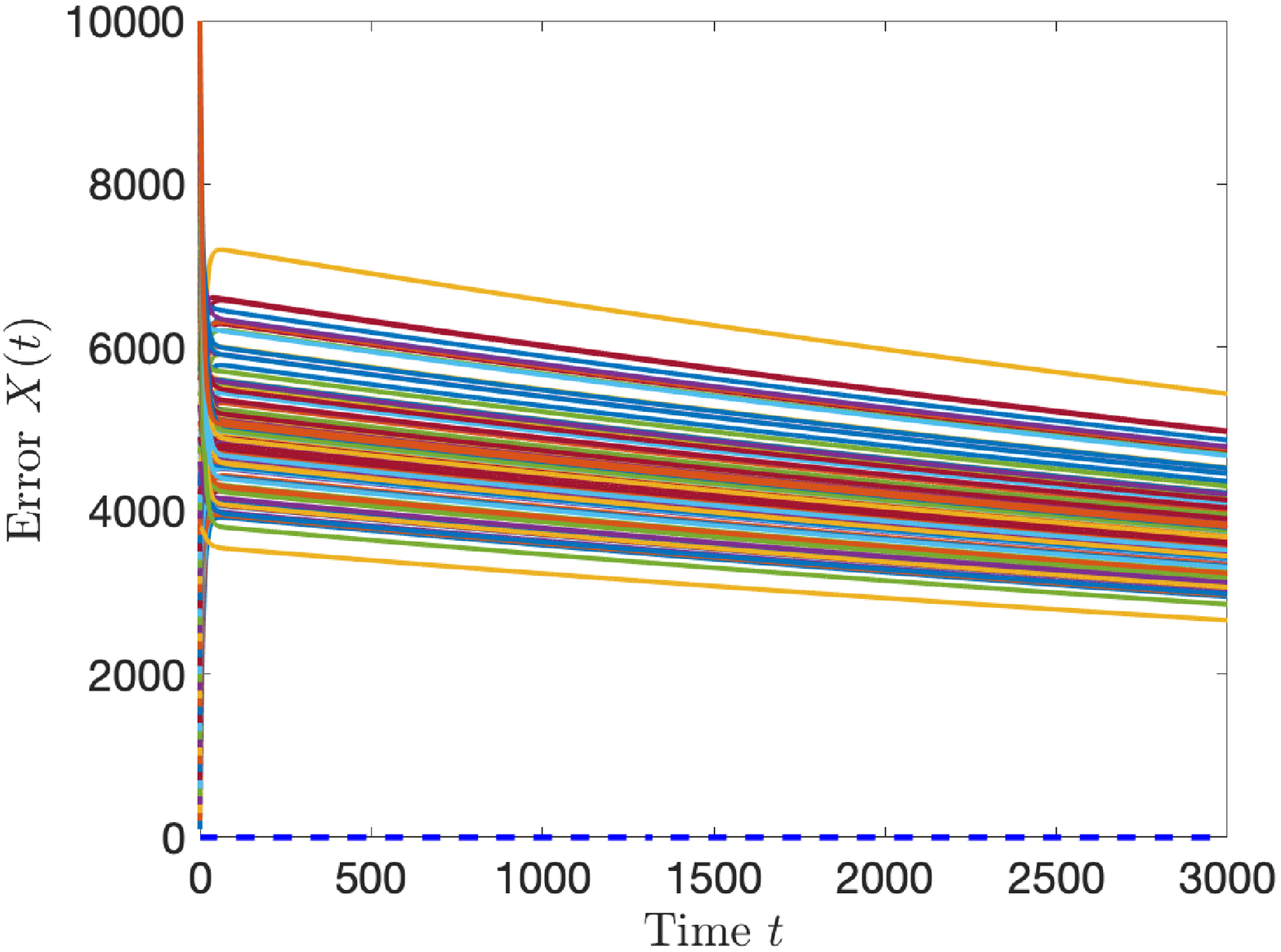}
\hspace{-.6cm}
\includegraphics[scale=0.4]{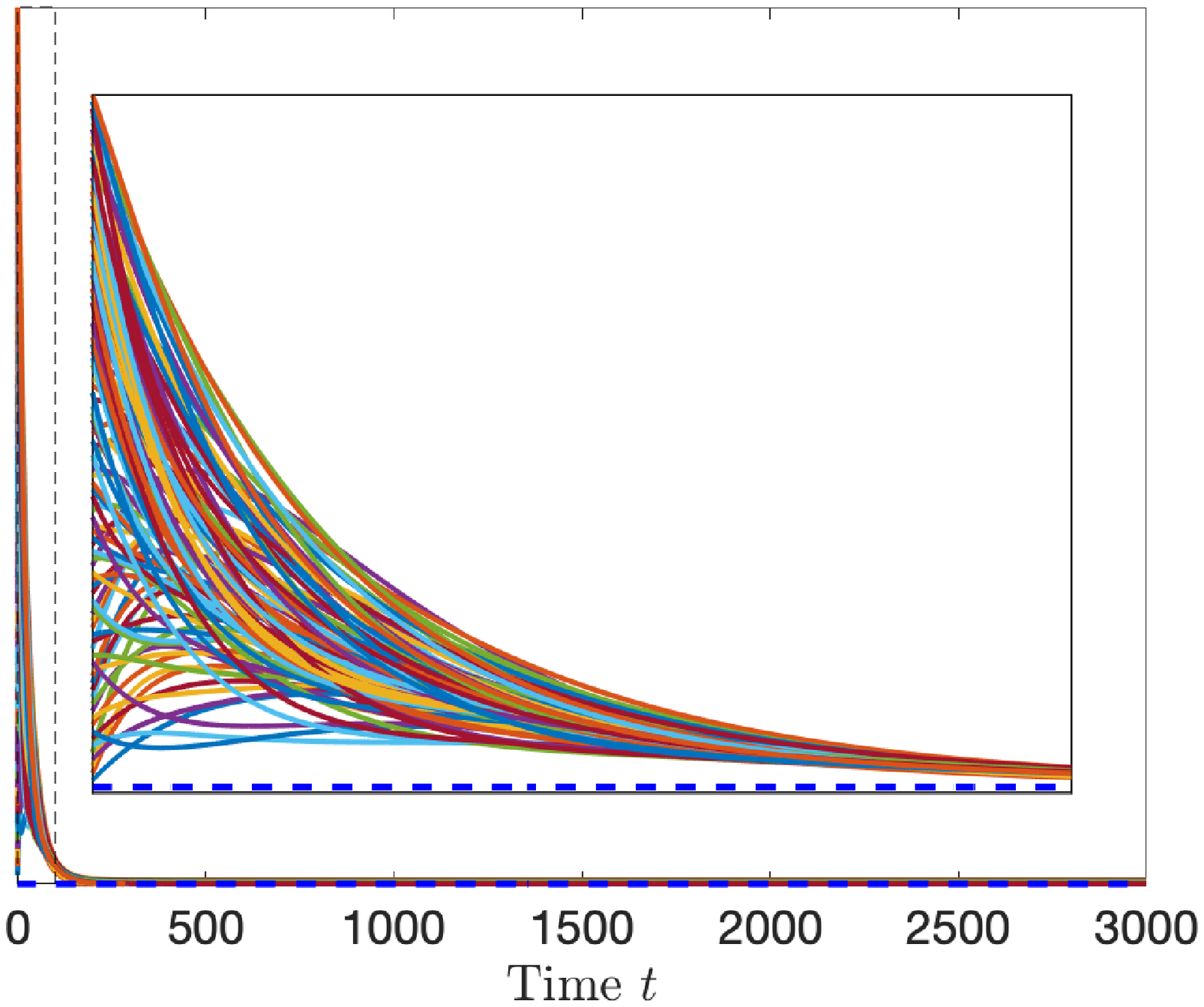}
\end{tabular}
\caption{Left: Dynamics is very slow and close to instability. 
Right: effect of the stabilizing control $u^{(2)}$ where $\hat{K}^{(r)}$ is from (\ref{LP1})-(\ref{LP4}).}
\label{fig:u2}
\end{figure}

\section{Conclusions.}
\label{sec:conc}

For a financial network we have developed a dynamic model of cascading failures. Theoretical analysis has provided a deeper understanding of the conditions that prevent cascades to occur and lead to  stability of the market values around certain equilibrium points. Such conditions involve the structure of the cross-holding matrix as well as the magnitude of the external assets. Results are illustrated on a large network with different connectivity. 
Finally, we have also provided a control design method to optimize the market investments which are now viewed as feedback-feedforward control inputs.

%
%
%

\section*{Acknowledgments.}
We thank Prof. Dr. Rakesh Vohra (University of Pennsylvania), for the useful discussions. In particular we thank him for his insights to improve the explanation of the motivation and novelty of our work, for the inputs to clarify the assumptions and conditions of our model, and for his comments about the control design.



\bibliographystyle{informs2014}

\end{document}